\documentclass[nopreprintline]{elsarticle}
\usepackage[utf8]{inputenc}
\usepackage[T1]{fontenc}
\usepackage{latexsym,amssymb,amsfonts,amsmath}
\usepackage{mathtools}
\usepackage{amsthm}
\usepackage{graphicx}
\usepackage[all]{xy}
\usepackage{natbib}
\large
\usepackage[dvipsnames]{xcolor}
\usepackage{mathrsfs} 

\usepackage{tikz}
\usetikzlibrary{arrows, decorations.markings, shapes, patterns}
\usepackage{tkz-berge}

\usepackage{caption}
\usepackage{subcaption}

\usepackage{comment}

\theoremstyle{plain}
\newtheorem{proposition}{Proposition}
\newtheorem{lemma}[proposition]{Lemma}
\newtheorem{theorem}[proposition]{Theorem}
\newtheorem{corollary}[proposition]{Corollary}
\newtheorem{case}{Case}[proposition]
\newtheorem{subcase}{Case}
\numberwithin{subcase}{case}

\theoremstyle{definition}

\theoremstyle{remark}
\newtheorem{claim}{Claim}[proposition]

\usepackage{lineno}
\DeclareMathOperator{\len}{length}

\DeclareMathOperator{\diam}{diam}

\title{Redicoloring some classes of circulant torunaments}

\author{Narda Cordero-Michel\corref{cor1}}
\ead{narda@ciencias.unam.mx}

\author{Mika Olsen}
\ead{olsen@cua.uam.mx}

\cortext[cor1]{Corresponding author.}

\address{Departamento de Matem\'{a}ticas Aplicadas y Sistemas, \\ Universidad Aut\'{o}noma Metropolitana - Cuajimalpa, M\'{e}xico}

\begin{document}

\begin{abstract}
Given a digraph $D$ with no loops, the \textit{dicoloring graph} of $D$, denoted by $\mathcal{D}_k(D)$, is the graph whose vertices
are the acyclic $k$-colorings of $D$ and two colorings are adjacent in $\mathcal{D}_k(D)$ if they differ in color on exactly one vertex. In this paper, we prove that there is no expression $\phi(\vec{\chi})$ in terms of the dichromatic number $\vec\chi$, such that the graph $\mathcal{D}_k(D)$ is connected for all graphs $D$ and integers $k\geq \phi(\vec\chi)$. We give conditions for the dicoloring graph of two infinite families of circulant tournaments to be connected and we provide upper bounds for its diameter. In particular, for the Payley tournament $\vec{C}_{7}(1,2,4)$, also known as $ST_7$, we prove that $\mathcal{D}_k(\vec{C}_{7}(1,2,4))$ is connected and has diameter 8, for each $k\geq 3$.
\end{abstract}

\maketitle

\section{Introduction}

Cereceda (\cite{Cereceda2007}, 2007) focused his doctoral research on the subject of graph recolouring of a $k$-colorable graph $G$, which is a reconfiguration problem were it is analyzed if a $k$-coloring of $G$ might be transformed into another $k$-coloring of $G$ by recoloring one vertex at a time in such a way that a proper $k$-coloring is maintained at any step. This problem is represented by the $k$-color graph (or $k$-recoloring graph) $\mathcal{C}_k(G)$, it is the graph that has the set of proper $k$-colorings of $G$ as its vertex set and two $k$-colorings are adjacent whenever they differ in color on precisely one vertex of $G$.
He obtained significant results for this problem when $k=3$, he determined the complexity of deciding if the coloring graph is connected for $k\geq 3$ and, when $\mathcal{C}_k(G)$ is connected, he investigated the length of a shortest path in $\mathcal{C}_k(G)$ between two colorings. Ever since, many authors have studied graph recolorings, for a survey on this matter we refer the reader to \cite{Mynhardt2020}. 
Bousquet and his collaborators \cite{BOUSQUET2024103876} studied digraph redicolorings as a generalization of the graph recoloring problem. They introduced the concept of the $k$-dicoloring graph $\mathcal{D}_k(D)$ of a digraph $D$, it is the graph whose vertex set is the set of acyclic $k$-colorings of $D$, with two $k$-colorings being adjacent whenever they differ on the color of exactly one vertex.
Bousquet \textit{et al.} proved that it is PSPACE-complete to decide whether two given $k$-colorings of a digraph $D$ belong to the same connected component of $\mathcal{D}_k(D)$ even when $k=2$, when $D$ has maximum degree equal to 5, or when $D$ is an oriented planar graph with maximum degree equal to 6.
Moreover, they also proved sufficient conditions for $\mathcal{D}_k(D)$ being connected in terms of some parameters of the maximum degree and minimum degree in $D$, they provided upper bounds for the diameter of $\mathcal{D}_k(D)$, as well as a lower bound for the number of edges in an oriented graph whose dicoloring graph has isolated vertices, and they established some conjectures. In \cite{Picasarri20245}, Picasarri-Arrieta proved sufficient conditions for $\mathcal{D}_k(D)$ being connected in terms of the maximum semi-degree in $D$ and he gave upper bounds for the diameter of the dicoloring graph. In \cite{NISSE2024191}, Nisse \textit{et. al.} obtained bounds for the diameter of $\mathcal{D}_k(D)$ in terms of the cycle degeneracy of $D$, the digrundy number of $D$, and the $\mathcal{D}$-width of $D$. Furthermore, they showed a connection between the dicoloring graph of a digraph $D$ and the color graph of its underlying graph $UG(D)$.

In \cite{Cereceda2008}, Cereceda proved that there is no expression $\phi(\chi)$ in terms of the chromatic number $\chi$, such that the graph $\mathcal{C}_k(G)$ is connected for all graphs $G$ and integers $k\geq \phi(\chi)$. 
In this direction, we will see that there is no expression $\phi(\vec{\chi})$ in terms of the dichromatic number $\vec{\chi}$, such that the graph $\mathcal{D}_k(D)$ is connected for all digraphs $D$ and integers $k\geq \phi(\vec{\chi})$. 
He also proved that if $k=\chi(G)\le3$, then $\mathcal{C}_k(G)$ is disconnected and if $k=\chi(G)\ge4$, then $\mathcal{C}_k(G)$ may be connected or disconnected. In this direction, we will see that there is an infinite family of 2-mixing 2-dichromatic tournaments, an infinite family of 3-mixing 3-dichromatic tournaments and a 3-dichromatic tournament which is not 3-mixing.

In \cite{Picasarri20245}, Picasarri-Arrieta proved for every digraph $D$ with $\Delta \geq 3$, where $\Delta=\max\{\max(d^+(v), d^-(v))\colon v \in V(D)\}$, and every $k \geq \Delta + 1$, that $\mathcal{D}_k(D)$ consists of isolated vertices and at most one further component that has diameter at most $c_\Delta n^2$, where $c_\Delta = O(\Delta^2)$ is a constant depending only on $\Delta$.
In this paper, we improve this result for two families of circulant tournaments $\vec{C}_{2n+1}\langle \emptyset \rangle$ and $\vec{C}_{2n+1}\langle n\rangle$ defined in the following section. We 
show for $n\geq 2$ and $ k\geq 3$ that $\mathcal{D}_{k}(\vec{C}_{2n+1}\langle \emptyset \rangle)$ is connected and has diameter at most $4n+1+\lfloor \frac{n+1}{2} \rfloor$; we also prove for $n\geq 4$ and $ k\geq 3$ that $\mathcal{D}_{k}(\vec{C}_{2n+1}\langle n \rangle)$ is connected and has diameter at most $4n+2+\lfloor \frac{n}{2} \rfloor$, whenever $k\neq \frac{2n+1}{3}$; and $\mathcal{D}_{k}(\vec{C}_{2n+1}\langle n \rangle)$ consists of isolated vertices and one further component that has diameter at most $2n+1+\lfloor\frac{2n+1}{4}\rfloor$ whenever $k=\frac{2n+1}{3}$.
    
Along this paper we will work with families of digraphs without digons, also called \emph{oriented graphs}, whose dichromatic numbers are known. We omit ``whithout digons'' from the hypotheses and we use that any pair of vertices induces an acyclic set without mention it.

\section{Definitions}

An \textit{acyclic $k$-coloring} of a digraph $D$ is a function $\alpha \colon V(D)\to \{1,2,\ldots,k\}$ such that each \textit{color class}, not necessarily nonempty, $C_i^\alpha=\{x\in V(D)\colon \alpha(x)=i\}$ with $i\in\{1,2,\ldots,k\}$, induces an acyclic subdigraph of $D$.
The \textit{dichromatic number} of a digraph $D$, denoted by $\vec{\chi}(D)$, is the minimum number $k$ for which $D$ admits an acyclic $k$-coloring of its vertex set, and we say that $D$ is \textit{$\vec{\chi}(D)$-dichromatic}. In this paper, every coloring is acyclic, so we will omit the word ``acyclic''.

For any $k\geq \vec{\chi}(D)$, the \textit{$k$-dicoloring graph of $D$}, denoted by $\mathcal{D}_k(D)$, is defined as the graph whose vertices are the $k$-colorings of $D$, and two  $k$-colorings of $D$ are \textit{adjacent} whenever they differ on precisely one vertex. Two  colorings of $D$, $\alpha$ and $\beta$, are at \textit{distance} $d$, denoted by $d(\alpha,\beta)$, if there exists a sequence of  colorings of $D$ $\alpha=\alpha_0, \alpha_1, \ldots, \alpha_d=\beta$ such that $\alpha_i$ and $\alpha_{i+1}$ are adjacent for each $i\in\{0,1,\ldots,d-1\}$ and any other sequence with this property has at least the same number of elements. Notice that, if there exists a $k$-coloring of $D$, then it has at least $\frac{k!}{(k-p)!}$ different $k$-colorings, where $p$ is the number of nonempty color classes. The digraph $D$ is \textit{$k$-mixing} if $\mathcal{D}_k(D)$ is connected and \textit{$k$-freezable} if $\mathcal{D}_k(D)$ contains an isolated vertex. 

A digraph $D$ is \textit{uniquely $k$-colorable} for some $k$, $k\geq 2$, whenever every $k$-coloring of $D$ induces the same partition of $V(D)$, this is when any two $k$-colorings of $D$ differ only in a permutation of the colors used for the same set of color classes. 

A digraph $D$ on $n$ vertices is always $(n+1)$-mixing as: 1) any $k$-coloring $\alpha$, with $k\leq n+1$, can be converted into a $(n+1)$-coloring $\beta$ simply by changing the colors of vertices, one by one, in such a way that each color class in $\beta$ is a singleton; 2) if $\beta$ and $\gamma$ are two $(n+1)$-colorings of $D$ such that each of their color classes is a singleton, then $\beta$ can be converted into $\gamma$ in the following way, for each vertex $x\in V(D)$, if $x$ is not colored with $\gamma(x)$ and $\gamma(x)$ is not currently assigned to any vertex, then recolor $x$ with $\gamma(x)$, and if $\gamma(x)$ is currently assigned to $y$, then first recolor $y$ with the color that is not currently assigned to any vertex, and then recolor $x$ with $\gamma(x)$. Hence, any two $k$-colorings, with $k\leq n+1$, can be converted into each other by recoloring one vertex at a time.

It is immediate from the definition of a dicoloring graph that if $D$ is a digraph with digons, then its 2-dicoloring graph $\mathcal{D}_2(D)$ is disconnected (see example in Figure \ref{fig: disconnected dicoloring graph}).

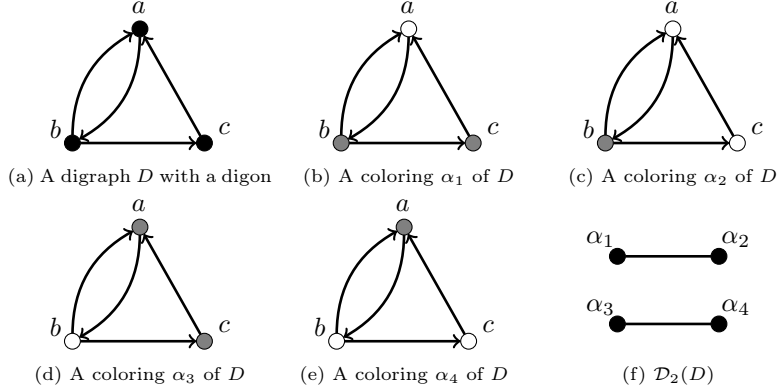
\begin{figure}[h!]
\centering
\begin{subfigure}[b]{0.29\textwidth}
\centering
\begin{tikzpicture}[x=0.45cm, y=0.45cm]
\tikzstyle{VertexStyle} = [shape = circle, fill = black, 
minimum size = 6pt, inner sep = 2pt, draw]
\SetVertexNoLabel
\Vertex[x=-4.6, y=-0.55]{b};
\Vertex[x=-0.7, y=-0.55]{c};
\Vertex[x=-2.605, y=2.824]{a};

\draw[color=black] (-5.1,-0.12) node {$b$};
\draw[color=black] (-0.12,-0.12) node {$c$};
\draw[color=black] (-2.64,3.46) node {$a$};

\draw [line width=1pt, black, ->] (b)--(c);
\draw [line width=1pt, black, ->] (c)--(a);
\draw[line width=1pt, black, ->] (a) to[bend left] (b);
\draw[line width=1pt, black, ->] (b) to[bend left] (a);
\end{tikzpicture}
\caption{A digraph $D$ with a digon}
\label{fig: 2-dichromatic digraph}   
\end{subfigure}
\begin{subfigure}[b]{0.28\textwidth}
\centering
\begin{tikzpicture}[x=0.45cm, y=0.45cm]
\tikzstyle{VertexStyle} = [shape = circle, fill = black, 
minimum size = 6pt, inner sep = 2pt, draw]
\SetVertexNoLabel

\tikzstyle{VertexStyle} = [shape = circle, fill = white, 
minimum size = 6pt, inner sep = 2pt, draw]
\SetVertexNoLabel
\Vertex[x=-2.605, y=2.824]{a};

\tikzstyle{VertexStyle} = [shape = circle, fill = gray, 
minimum size = 6pt, inner sep = 2pt, draw]
\SetVertexNoLabel
\Vertex[x=-4.6, y=-0.55]{b};
\Vertex[x=-0.7, y=-0.55]{c};

\draw[color=black] (-5.1,-0.12) node {$b$};
\draw[color=black] (-0.12,-0.12) node {$c$};
\draw[color=black] (-2.64,3.46) node {$a$};

\draw [line width=1pt, black, ->] (b)--(c);
\draw [line width=1pt, black, ->] (c)--(a);
\draw[line width=1pt, black, ->] (a) to[bend left] (b);
\draw[line width=1pt, black, ->] (b) to[bend left] (a);
\end{tikzpicture}
\caption{A coloring $\alpha_1$ of $D$ }
\label{fig: acyclic 2-coloring 1}    
\end{subfigure}
\begin{subfigure}[b]{0.28\textwidth}
\centering
\begin{tikzpicture}[x=0.45cm, y=0.45cm]
\tikzstyle{VertexStyle} = [shape = circle, fill = black, 
minimum size = 6pt, inner sep = 2pt, draw]
\SetVertexNoLabel

\tikzstyle{VertexStyle} = [shape = circle, fill = white, 
minimum size = 6pt, inner sep = 2pt, draw]
\SetVertexNoLabel
\Vertex[x=-2.605, y=2.824]{a};
\Vertex[x=-0.7, y=-0.55]{c};

\tikzstyle{VertexStyle} = [shape = circle, fill = gray, 
minimum size = 6pt, inner sep = 2pt, draw]
\SetVertexNoLabel
\Vertex[x=-4.6, y=-0.55]{b};

\draw[color=black] (-5.1,-0.12) node {$b$};
\draw[color=black] (-0.12,-0.12) node {$c$};
\draw[color=black] (-2.64,3.46) node {$a$};

\draw [line width=1pt, black, ->] (b)--(c);
\draw [line width=1pt, black, ->] (c)--(a);
\draw[line width=1pt, black, ->] (a) to[bend left] (b);
\draw[line width=1pt, black, ->] (b) to[bend left] (a);
\end{tikzpicture}
\caption{A coloring $\alpha_2$ of $D$}
\label{fig: acyclic 2-coloring 2}    
\end{subfigure}

\begin{subfigure}[b]{0.28\textwidth}
\centering
\begin{tikzpicture}[x=0.45cm, y=0.45cm]
\tikzstyle{VertexStyle} = [shape = circle, fill = white, 
minimum size = 6pt, inner sep = 2pt, draw]
\SetVertexNoLabel
\Vertex[x=-4.6, y=-0.55]{b};

\tikzstyle{VertexStyle} = [shape = circle, fill = gray, 
minimum size = 6pt, inner sep = 2pt, draw]
\SetVertexNoLabel
\Vertex[x=-2.605, y=2.824]{a};
\Vertex[x=-0.7, y=-0.55]{c};

\draw[color=black] (-5.1,-0.12) node {$b$};
\draw[color=black] (-0.12,-0.12) node {$c$};
\draw[color=black] (-2.64,3.46) node {$a$};

\draw [line width=1pt, black, ->] (b)--(c);
\draw [line width=1pt, black, ->] (c)--(a);
\draw[line width=1pt, black, ->] (a) to[bend left] (b);
\draw[line width=1pt, black, ->] (b) to[bend left] (a);
\end{tikzpicture}
\caption{A coloring $\alpha_3$ of $D$}
\label{fig: acyclic 2-coloring 3}    
\end{subfigure}
\begin{subfigure}[b]{0.28\textwidth}
\centering
\begin{tikzpicture}[x=0.45cm, y=0.45cm]
\tikzstyle{VertexStyle} = [shape = circle, fill = white, 
minimum size = 6pt, inner sep = 2pt, draw]
\SetVertexNoLabel
\Vertex[x=-4.6, y=-0.55]{b};
\Vertex[x=-0.7, y=-0.55]{c};

\tikzstyle{VertexStyle} = [shape = circle, fill = gray, 
minimum size = 6pt, inner sep = 2pt, draw]
\SetVertexNoLabel
\Vertex[x=-2.605, y=2.824]{a};

\draw[color=black] (-5.1,-0.12) node {$b$};
\draw[color=black] (-0.12,-0.12) node {$c$};
\draw[color=black] (-2.64,3.46) node {$a$};

\draw [line width=1pt, black, ->] (b)--(c);
\draw [line width=1pt, black, ->] (c)--(a);
\draw[line width=1pt, black, ->] (a) to[bend left] (b);
\draw[line width=1pt, black, ->] (b) to[bend left] (a);
\end{tikzpicture}
\caption{A coloring $\alpha_4$ of $D$}
\label{fig: acyclic 2-coloring 4}    
\end{subfigure}
\begin{subfigure}[b]{0.28 \textwidth}
\centering
\begin{tikzpicture}[x=0.45cm, y=0.45cm]
\tikzstyle{VertexStyle} = [shape = circle, fill = black, 
minimum size = 6pt, inner sep = 2pt, draw]
\SetVertexNoLabel
\Vertex[x=0, y=3]{a1};
\Vertex[x=3, y=3]{a2};
\Vertex[x=0, y=1]{a3};
\Vertex[x=3, y=1]{a4};

\draw[color=black] (-0.5,3.5) node {$\alpha_1$};
\draw[color=black] (3.5,3.5) node {$\alpha_2$};
\draw[color=black] (-0.5,1.5) node {$\alpha_3$};
\draw[color=black] (3.5,1.5) node {$\alpha_4$};
\draw[color=black] (3.5,0.5) node {};

\draw [line width=1pt, black] (a1)--(a2);
\draw [line width=1pt, black] (a3)--(a4);
\end{tikzpicture}
\caption{$\mathcal{D}_2(D)$}
\label{fig: dicoloring graph}   
\end{subfigure}

\caption{A 2-dichromatic digraph $D$, with a disconnected 2-dicoloring graph.}
\label{fig: disconnected dicoloring graph} 
\end{figure}

For $n\geq 2$, the \textit{cyclic circulant tournament}, denoted by $\vec{C}_{2n+1}(1,2, \ldots, n)$, is the digraph whose vertex set is $V (\vec{C}_{2n+1}(1,2, \ldots, n)) = \mathbb{Z}_{2n+1}$ and its arc set consists of all the arcs $(a, a+j)$ for every jump $j \in \{1, 2, \ldots, n\}$ and every $a \in \mathbb{Z}_{2n+1}$. 

The circulant tournament obtained from $\vec{C}_{2n+1}(1,2, \ldots, n)$ by reversing one of its jumps, say $j$, is denoted by $\vec{C}_{2n+1}\langle j\rangle$, this is, for a fixed $j$ the set of arcs $\{(a, a+j)\colon a\in\mathbb{Z}_{2n+1}\}$ of $\vec{C}_{2n+1}(1,2, \ldots, n)$ is replaced by the set of arcs  $\{(a, a - j)\colon a\in\mathbb{Z}_{2n+1}\}$ in $\vec{C}_{2n+1}\langle j\rangle$. In this way, we can denote the cyclic circulant tournament $\vec{C}_{2n+1}(1,2, \ldots, n)$ by $\vec{C}_{2n+1}\langle \emptyset \rangle$.

\section{Preliminaries}

First of all we establish an obvious proposition that has consequences for a family of digraphs studied by Neumann-Lara and his collaborators in \cite{NEUMANNLARA198665}. 

\begin{proposition}
    If a digraph $D$ is uniquely $k$-colorable, then $D$ is $k$-freezable and $\mathcal{D}_k(D)$ is the graph with $k!$ isolated vertices.
\end{proposition}

In \cite{NEUMANNLARA198665}, Neumann-Lara \textit{et al.} proved that there exists an infinite family of uniquely colorable $k$-dichromatic tournaments, for all $k \geq 2$. They also proved that the tournament $\vec{C}_{2n+1}\langle n\rangle - \{0\}$ is uniquely 2-colorable, for $n \geq 4$, with chromatic classes $\{1, 2, \ldots, n\}$ and $\{n+1, n+2, \ldots, 2n\}$. 

\begin{corollary}
    For each $k\ge 2$ there exists a tournament $T$ such that $T$ is $k$-freezable and it's dicoloring graph has $k!$ isolated vertices. In particular, for all $n \geq 4$, $\vec{C}_{2n+1}\langle n\rangle - \{0\}$ is 2-freezable and $\mathcal{D}_2(\vec{C}_{2n+1}\langle n\rangle - \{0\})$ consists of two isolated vertices.
\end{corollary}

In the following two lemmas, we analyze some particular cases where at least one of two colorings has a lot of singular color classes, in these cases it is easier to determine the distance between the colorings.

\begin{lemma}
\label{lemma: k colors k vertices}
Let $D$ be a digraph of order $n$ and $k$ and $l$ integers with $2\leq k\leq \min\{l,n\}$. And let $\alpha$ and $\beta$ be two $l$-colorings of $D$, where $\beta$ has $k$ singular color classes, $C_1^\beta$, $C_2^\beta$, \ldots, $C_k^\beta$, and $\alpha(u)=\beta(u)$ for each $u\in V(D)\setminus\left(\bigcup_{i=1}^k C_i^\beta\right)$. Then $d(\alpha, \beta)\leq k$ in $\mathcal{D}_n(D)$.
\end{lemma}
\begin{proof}
    Proceeding by induction on $k$.
    
    Whenever $k=2$, we have that $C_1^\alpha \cup C_2^\alpha (\subseteq C_1^\beta \cup C_2^\beta$. Start at $\alpha$ and recolor the vertices in $C_1^\beta \cup C_2^\beta$ with the color assigned to them by $\beta$, this takes at most two steps since $D$ has no digon and any pair of vertices induces an acyclic set. 

    Assume that the result is valid for $k\geq 2$. And suppose that $\beta$ has $k+1$ singular color classes $C_1^\beta$, $C_2^\beta$, \ldots, $C_{k+1}^\beta$, and $\alpha(u)=\beta(u)$ for each $u\in V(D)\setminus\left(\bigcup_{i=1}^{k+1} C_i^\beta\right)$. Hence, $\bigcup_{i=1}^{k+1} C_i^\alpha\subseteq\bigcup_{i=1}^{k+1} C_i^\beta$.

    Start at $\alpha$. If there is an index $j\in \{1,2,\ldots, k+1\}$ such that $C_j^\alpha=\emptyset$, then recolor the unique vertex $v$ in $C_j^\beta$ with color $\beta(v)$. Call the new coloring $\alpha'$ and let $J=\{1,2,\ldots,k+1\}\setminus\{j\}$. The coloring $\alpha'$  satisfies that $d(\alpha,\alpha')=1$, $\alpha'(u)=\beta(u)$ for each vertex in $V(D)\setminus\left(\bigcup_{i\in J} C_i^\beta\right)$ and $|C_i^\beta|=1$ for each $i\in J$. So, by the inductive hypothesis, $d(\alpha',\beta)\leq k$. And thus, $d(\alpha,\beta)\leq k+1$. 

    If $C_i^\alpha\neq \emptyset$ for each $i \in \{1,2,\ldots, k+1\}$, then $C_i^\alpha$ is a singleton for each $i \in \{1,2,\ldots, k+1\}$. Assume that $C_i^\alpha\neq C_i^\beta$ for each $i \in \{1,2,\ldots, k+1\}$. Call $u_i$ the vertex in $C_i^\alpha$. Recolor vertex $u_1$ with color $\beta(u_1)$, then recolor vertex $u_{\beta(u_1)}$ with color $\beta(u_{\beta(u_1)})$, continue in this way until each vertex has the color that $\beta$ assigns to it, this takes $k+1$ steps, and each intermediary step is a coloring since $D$ has no digon. Then, $d(\alpha,\beta)\leq k+1$.
\end{proof}

\begin{lemma}
\label{lemma: k colors k+1 vertices}
Let $D$ be a digraph of order $n$  , and $k$ and $l$ integers with $2\leq k\leq \min\{l, n-1\}$. And let $\alpha$ and $\beta$ be two $l$-colorings of $D$, where $\beta$ has a color class with two vertices, $C_1^\beta$, $k-1$ singular color classes, $C_2^\beta$, $C_3^\beta$, \ldots, $C_{k}^\beta$, and $\alpha(u)=\beta(u)$ for each $u\in V(D)\setminus\left(\bigcup_{i=1}^k C_i^\beta\right)$. Then $d(\alpha, \beta)\leq k+1$ in $\mathcal{D}_n(D)$.
\end{lemma}
\begin{proof}
    Proceeding by induction on $k$.
    
    Whenever $k=2$, $C_1^\beta=\{u,v\}$ and $C_2^\beta= \{w\}$ and $C_1^\alpha \cup C_2^\alpha \subseteq \{u,v,w\}$.  

    Start at $\alpha$. If $C_i^\alpha=\emptyset$, then $C_{3-i}^\alpha\subseteq\{u,v,w\}$ and we can recolor the vertices in $\{u,v,w\}$ to get $\beta$ in at most 3 steps. 

    If $|C_i^\alpha|=1$, then $|C_{3-i}^\alpha|=2$. Assuming $\alpha\neq \beta$, we can recolor one vertex in $C_{3-i}^\alpha$ with the color assigned to it by $\beta$, namely $i$. Then, if needed, recolor the vertex in $C_i^\alpha$ with color $3-i$. And finally, recolor the last vertex, if needed, with color $i$. This takes at most $3$ steps.

    Assume that the result is valid for $k\geq 2$. And suppose that $\beta$ has a color class with two vertices, namely $C_1^\beta$, $k$ singular color classes, namely $C_2^\beta$, $C_3^\beta$, \ldots, $C_{k+1}^\beta$, and $\alpha(u)=\beta(u)$ for each $u\in V(D)\setminus\left(\bigcup_{i=1}^{k+1} C_i^\beta\right)$. Hence, $\bigcup_{i=1}^{k+1} C_i^\alpha \subseteq  \bigcup_{i=1}^{k+1} C_i^\beta$.

    Start at $\alpha$. If $C_1^\alpha=\emptyset$, then recolor the two vertices in $C_1^\beta$ with color 1. The new coloring $\alpha'$ satisfies $d(\alpha,\alpha')=2$ and, by Lemma \ref{lemma: k colors k vertices},  $d(\alpha',\beta)\leq k$. Hence, $d(\alpha,\beta)\leq k+2$. Assume that $C_1^\alpha\neq\emptyset$ 
    
    If there is an index $j\in \{2,3,\ldots, k+1\}$ such that $C_j^\alpha=\emptyset$, then recolor the unique vertex $v$ in $C_j^\beta$ with color $\beta(v)$. Call the new coloring $\alpha'$ and let $J=\{2,3,\ldots,k+1\}\setminus\{j\}$. The coloring $\alpha'$ satisfies that $d(\alpha,\alpha')=1$, $\alpha'(u)=\beta(u)$ for each vertex in $V(D)\setminus\left(\bigcup_{i\in \{1\}\cup J} C_i^\beta\right)$ and $|C_i^\beta|=1$ for each $i\in J$. So, by the inductive hypothesis, $d(\alpha',\beta)\leq k+1$. And thus, $d(\alpha,\beta)\leq k+2$. 

    If $C_i^\alpha\neq \emptyset$ for each $i \in \{1,2,\ldots, k+1\}$, then $C_j^\alpha$ has precisely two elements for some $j\in \{1,2,\ldots, k+1\}$, and $C_i^\alpha$ is a singleton for each $i \in J$, where $J=\{1,2,\ldots, k+1\}\setminus\{j\}$. If $j=1$ and $C_1^\alpha=C_1^\beta$, then $d(\alpha,\beta)\leq k$, by Lemma \ref{lemma: k colors k vertices}. So, we assume this is not the case. Necessarily, one of the vertices in $C_j^\alpha$, say $v$, can be recolored with the color assigned to it by $\beta$, so we recolor $v$ with $\beta(v)$. The new coloring $\alpha'$ satisfies that $C_{\beta(v)}^{\alpha'}$ has precisely two elements. If the other vertex in $C_{\beta(v)}^{\alpha'}$, say $w$, has $\alpha'(w)=\beta(w)$, then $d(\alpha',\beta)\leq k$, by Lemma \ref{lemma: k colors k vertices}. If $\alpha'(w)\neq\beta(w)$, then we recolor vertex $w$ with $\beta(w)$. Proceeding this way, eventually, we get $\beta$ by recoloring  each vertex in $\bigcup_{i=1}^{k+1} C_i^\beta$ at most once. Hence, $d(\alpha,\beta)\leq k+2$. This concludes the proof.
\end{proof}

\begin{corollary}
    Any digraph $D$ of order $n$ is $k$-mixing for every $k\geq n-1$ and $\mathcal{D}_k(D)$ has diameter at most $2n$.
    \label{corollary: n-mixing}
\end{corollary}

\section{The cyclic circulant tournament $\vec{C}_{2n+1}\langle \emptyset \rangle$}

The following result is an extended version of a Theorem 1 in  \cite{NEUMANNLARA198483} that will help us to prove the result announced at the introduction.

\begin{theorem}[Neumann-Lara and Urrutia \cite{NEUMANNLARA198483}]
\label{teo: Neumann-Lara} Let $n$ be a positive integer. Then $\vec{\chi}(\vec{C}_{2n+1}\langle \emptyset \rangle)=2$ and every 2-coloring of $\vec{C}_{2n+1}\langle \emptyset \rangle$ induces a partition of $V(\vec{C}_{2n+1}\langle \emptyset \rangle)$ into two sets of the form $\{a, a+1, a+2, \ldots,a+n\}$ and $\{a+n+1,a+n+2,  \ldots,a-1\}$, with $a\in\mathbb{Z}_{2n+1}$. Moreover, if $w$, $x$ and $y$ are three different vertices in $\mathbb{Z}_{2n+1}$ satisfying that $\{w,x,y\}\not \subseteq \{a, a+1, a+2, \ldots,a+n\}$, for each $a\in \mathbb{Z}_{2n+1}$, then $\{w,x,y\}$ induces a triangle in $\vec{C}_{2n+1}\langle \emptyset \rangle$.
\end{theorem}

\begin{proposition}
\label{2-recoloring graph of the cyclic circulant tournament}
    The 2-dicoloring graph of the cyclic circulant tournament {$\vec{C}_{2n+1}\langle \emptyset \rangle$} is the cycle $C_{4n+2}$. Moreover, $\diam(\mathcal{D}_2(\vec{C}_{2n+1}\langle \emptyset \rangle))=2n+1$.
\end{proposition}
\begin{proof}
    Let $T=\vec{C}_{2n+1}\langle \emptyset \rangle$.
    For each $a\in \mathbb{Z}_{2n+1}$, let $\alpha_a$ and $\beta_a$ be the 2-colorings of $T$ given by $$\alpha_a(i)=\begin{cases}
        1&\text{if }i\in\{a, a+1, a+2, \ldots,a+n\};\\
        2&\text{if }i\in\{a+n+1,a+n+2,  \ldots,a-1\}.
    \end{cases}$$ and $$\beta_a(i)=\begin{cases}
        1&\text{if }i\in\{a,a+1, a+2, \ldots,a+n-1\};\\
        2&\text{if }i\in\{a+n,a+n+1,  \ldots,a-1\}.
    \end{cases}$$
    Since every 2-coloring of $T$ induces a partition of $V(T)$ into two sets of the form $\{a, a+1, a+2, \ldots,a+n\}$ and $\{a+n+1,a+n+2,  \ldots,a-1\}$, any 2-coloring of $T$ corresponds to $\alpha_a$ or $\beta_a$ for some $a\in \mathbb{Z}_{2n+1}$. Hence, $V(\mathcal{D}_2(T))=\{\alpha_a\colon a\in \mathbb{Z}_{2n+1}\}\cup \{\beta_a\colon a\in \mathbb{Z}_{2n+1}\}$.
    In the coloring $\alpha_a(i)$ the only vertices whose color can be changed are the vertices $a$ and $a+n$, thus the vertex $\alpha_a(i)$ has  degree 2 in the graph $\mathcal{D}_2(T)$. Analogously, $\beta_a(i)$ has degree 2 in the graph $\mathcal{D}_2(T)$.
    
\begin{claim}
\label{claim:neighbors}
\begin{sloppypar}
    For each $a\in \mathbb{Z}_{2n+1}$, $N_{\mathcal{D}_2(T)}(\alpha_a)=\{\beta_a,\beta_{a+1}\}$ and $N_{\mathcal{D}_2(T)}(\beta_a)=\{\alpha_{a-1},\alpha_{a}\}$.
\end{sloppypar}
    
    We have that $\alpha_a(a+n)=1\neq 2=\beta_a(a+n)$ and $\alpha_a(i)=\beta_a(i)$ whenever $i\neq a+n$, thus, $\alpha_a$ and $\beta_a$ are adjacent in $\mathcal{D}_2(T)$;  $\alpha_a(a)=1\neq 2=\beta_{a+1}(a)$ and $\alpha_a(i)=\beta_{a+1}(i)$ whenever $i\neq a$, thus, $\alpha_a$ and $\beta_{a+1}$ are adjacent in $\mathcal{D}_2(T)$.  
    Hence, $\{\beta_a,\beta_{a+1}\}\subseteq N_{\mathcal{D}_2(T)}(\alpha_a)$ and analogously, $\{\alpha_{a-1},\alpha_{a}\}\subseteq N_{\mathcal{D}_2(T)}(\beta_a)$. 
    Since each vertex in $\mathcal{D}_2(T)$ has degree equal to 2, it follows that $ N_{\mathcal{D}_2(T)}(\alpha_a)=\{\beta_a,\beta_{a+1}\}$ and $N_{\mathcal{D}_2(T)}(\beta_a)=\{\alpha_{a-1},\alpha_{a}\}$.  \vspace{1.7mm}
\end{claim}

From Claim \ref{claim:neighbors}, $W=\beta_0\alpha_0\beta_1\alpha_1\cdots \beta_{2n}\alpha_{2n}\beta_0$ is a Hamiltonian cycle in $\mathcal{D}_2(T)$  and each vertex in $\mathcal{D}_2(T)$ has degree 2. Therefore,  $\mathcal{D}_2(T)$ is a cycle of length $4n+2$.
\end{proof}

\begin{theorem}
    \label{theorem cyclic circulant tournament}
    Let $k\geq 3$. Then $\vec{C}_{2n+1}\langle \emptyset \rangle$ is $k$-mixing and $\mathcal{D}_k(\vec{C}_{2n+1}\langle \emptyset \rangle)$ has diameter at most $4n+1+\lfloor \frac{n+1}{2} \rfloor$. 
\end{theorem}
\begin{proof}
Let $T=\vec{C}_{2n+1}\langle \emptyset \rangle$.
Since $T$ is 2-dichromatic, $\mathcal{D}_{k}(T)$ is defined for every $k\geq 2$. 
Let $k\geq 3$. By Theorem \ref{teo: Neumann-Lara}, for any $k$-coloring $\alpha$ of $T$, and any color class $C_r^\alpha$ with $r\in\{1,2,\ldots,k\}$, there exists an $a\in\mathbb{Z}_{2n+1}$, such that $C_r^\alpha\subseteq \{a,a+1,\ldots, a+n\}$. 
Let $\alpha, \beta \colon \mathbb{Z}_{2n+1}\to \{1,2,\ldots,k\}$ be two $k$-colorings of $T$. And let $C_1^\alpha$, $C_2^\alpha$, \ldots, $C_k^\alpha$ and $C_1^\beta$, $C_2^\beta$, \ldots, $C_k^\beta$ be the color classes of $\alpha$ and $\beta$, respectively.

\begin{case}
$C_i^\alpha=\emptyset\mbox{ for some }i\in \{1,2,\ldots, k\}.$
\end{case}

Suppose w.l.o.g. that $C_k^\alpha=\emptyset$. Let $b\in\mathbb{Z}_{2n+1}$, such that $C_k^\beta\subseteq \{b,b+1,\ldots, b+n\}$. Assume w.l.o.g. that $b=0$, then $C_k^{\beta}\subseteq \{0,1,2,\ldots, n\}$. Start at $\alpha$ and recolor each vertex in $\{0,1,2,\ldots, n\}$ with color $k$, this takes $n+1$ steps. Let $\alpha'$ be the current coloring.
Recolor, when needed, each vertex $u$ in $\{n+1,n+2,\ldots, 2n\}$ with $\beta(u)$. 
Finally, recolor each vertex $u$ in $\{0,1,2,\ldots, n\}\setminus C_k^\beta$ with $\beta(u)$. The total process takes at most $n+1+n+n-|C_k^\beta|\leq 3n+1$.

\begin{case} There are two color classes $C_i^\alpha$ and $C_j^\alpha$, with $\{i,j\}\subset\{1,2,\ldots, k\}$ and $i\neq j$, such that $C_i^\alpha \cup C_j^\alpha \subseteq \{a,a+1,\ldots,a+n\}$ for some $a\in \mathbb{Z}_{2n+1}$.
\end{case}

Suppose w.l.o.g. that $|C_i^\alpha|\geq |C_j^\alpha|$. Start at $\alpha$, then recolor all vertices in $C_j^\alpha$ with color $i$. This process takes at most $\lfloor \frac{n+1}{2} \rfloor$ steps, and the new coloring $\alpha'$ satisfies that $C_j^{\alpha'}=\emptyset$, and thus $d(\alpha',\beta)\leq 3n+1$, by Case 1. Hence, $d(\alpha,\beta)\leq d(\alpha,\alpha')+d(\alpha',\beta) \leq 3n+1+\lfloor \frac{n+1}{2} \rfloor$.

\begin{case} For each pair of color classes $C_i^\alpha$ and $C_j^\alpha$, with $\{i,j\}\subset\{1,2,\ldots, k\}$ and $i\neq j$, there is no $a\in \mathbb{Z}_{2n+1}$ such that $C_i^\alpha \cup C_j^\alpha \subseteq \{a,a+1,\ldots,a+n\}$.
\end{case}

Consider a color class $C_i^\alpha$, with $i\in\{1,2,\ldots, k\}$, such that $C_i^\alpha\neq \emptyset$. 
Let $a\in \mathbb{Z}_{2n+1}$ such that $C_i^\alpha \subseteq \{a,a+1,\ldots,a+n\}$. 
Assume w.l.o.g. that $a=0$. So, if we start at $\alpha$ we can recolor every vertex in $\{0,1,2\ldots,n\}\setminus C_i^\alpha$ with color $i$, this process takes at most $n$ steps. Let $\alpha'$ be the current coloring. 
We have that $\alpha'$ satisfies the hypothesis of Case 2, so $d(\alpha',\beta)\leq 3n+1+\lfloor \frac{n+1}{2} \rfloor$, and thus $d(\alpha,\beta)\leq d(\alpha,\alpha')+d(\alpha',\beta) \leq 4n+1+\lfloor \frac{n+1}{2} \rfloor$. 

\end{proof}

\section{The circulant tournament $\vec{C}_{2n+1}\langle n \rangle$}

In what follows we will show a series of results in order to prove our main result on $\vec{C}_{2n+1}\langle n \rangle$ tournaments:

\begin{theorem}
Let $n$ and $k$ be two integers, $n\geq 3$ and $ k\geq 3$, and consider the cyclic tournament $\vec{C}_{2n+1}\langle n \rangle$. Then either:
\begin{enumerate}[(i)]
    \item $\vec{C}_{7}\langle 3 \rangle$ is $k$-mixing and  $\mathcal{D}_{k}(\vec{C}_{7}\langle 3 \rangle)$ has diameter at most $8$,
    \item for $n\geq 5$, $\vec{C}_{2n+1}\langle n \rangle$ is $3$-mixing and $\mathcal{D}_{3}(\vec{C}_{2n+1}\langle n \rangle)$ has diameter at most $3n+4+\lfloor \frac{2n+1}{3} \rfloor$, 
    \item for $n\geq 4$, $k\geq 4$ and $k\neq \frac{2n+1}{3}$, $\vec{C}_{2n+1}\langle n \rangle$ is $k$-mixing and $\mathcal{D}_{k}(\vec{C}_{2n+1}\langle n \rangle)$ has diameter at most $4n+2+\lfloor \frac{n}{2} \rfloor$, or 
    \item for $n\geq 4$ and $k= \frac{2n+1}{3}$, $\vec{C}_{2n+1}\langle n \rangle$ is not $k$-mixing, $\mathcal{D}_{k}(\vec{C}_{2n+1}\langle n \rangle)$ has $3k!$ isolated vertices and the distance between two non-frozen colorings $\alpha$ and $\beta$ is at most $2n+1+\lfloor \frac{2n+1}{4}\rfloor$.
\end{enumerate}
\label{theorem: main on n}
\end{theorem}

Theorem \ref{theorem: main on n} follows from Theorems \ref{teo: grafica de recoloracion conexa para k=3 y n volteada}, \ref{theorem: ST7}, \ref{theorem: D_k(n) k not multiple of 3} and \ref{theorem: D_k(n) k a multiple of 3}. In order to prove those theorems, we need the following four lemmas:

\begin{lemma}
    \label{lemma: acyclic triangle in circulant tournament with (n+1,0)}
    Let $\vec{C}_{2n+1}\langle n \rangle$ and $i\in \mathbb{Z}_{2n+1}$, then the subset $\{i, i+n, i+n+1\}$ of $\mathbb{Z}_{2n+1}$ induces a maximal acyclic subdigraph. 
\end{lemma}

\begin{proof}
    Let $T=\vec{C}_{2n+1}\langle n \rangle$.
    By the definition of $T$, we have that $(i+n,i)$, $(i+n, i+n+1)$, and $(i,i+n+1)$ are arcs in $T$. 
    Hence, the subdigraph induced by $\{i, i+n, i+n+1\}$ is acyclic.

    Now take $j$ in $\mathbb{Z}_{2n+1}\setminus \{i, i+n, i+n+1\}$ and let $D$ be the subdigraph of $T$ induced by $\{j, i, i+n, i+n+1\}$: 
    if $j\in\{i+1,1+2,\ldots, i+n-1\}$, then $(i,j)$ and $(j,i+n)$ are arcs in $D$, and thus $(i,j,i+n,i)$ is a cycle in $D$; and if $j\in\{i+n+2,i+n+3,\ldots, i-1\}$, then $(i+n+1,j)$ and $(j,i)$ are arcs in $D$, and thus $(i+n+1,j,i,i+n+1)$ is a cycle in $D$.   
\end{proof}

\begin{lemma}
\label{lemma: n volteada aciclico de n}
    If $V\subseteq \mathbb{Z}_{2n+1}$ induces an acyclic tournament in $\vec{C}_{2n+1}\langle n \rangle$, then $|V|\leq n$. Moreover, if $|V|=n$, then, for some $a\in\mathbb{Z}_{2n+1}$,  $V=\{a,a+1,\ldots, a+n-1\}$ or $V=\{a,a+1,\ldots, a+n+1\}\setminus\{a+1,a+n\}$. 
\end{lemma}
\begin{proof}
    Let $T=\vec{C}_{2n+1}\langle n \rangle$.
    To prove the first assertion, consider a set $V\subseteq \mathbb{Z}_{2n+1}$ which induces an acyclic tournament $D$ in $T$. 
    Then $V$ contains a source $a$, and thus $|V|\leq |\{a\}\cup N^+(a)|=n+1$. 
    Since $(a+1,a+n-1,a+n+1,a+1)$ is a directed triangle in $D$, it follows that $|V|\leq n$.

    Now, assume $|V|=n$. If $a+n+1\notin V$, then $V=\{a,a+1,\ldots,a+n-1\}$. If  $a+n+1\in V$, then $a+1\notin V$ or $(a+1,a+j,a+n+1,a+1)$ would be a directed triangle in $D$ for $1<j<n$. In this way, $V=\{a,a+1,\ldots,a+n+1\}\setminus\{a+1,a+n\}$.
\end{proof}

We call a subset $\{i, i+n, i+n+1\}$ of $\mathbb{Z}_{2n+1}$ in $\vec{C}_{2n+1}\langle n \rangle$ a \textit{forbidden triangle}. 

\begin{lemma}
    Let $n$ and $k$ be two integers, with $n\geq 4$ and $k\geq 3$. Let $\alpha\colon \mathbb{Z}_{2n+1}\to \{1,2,\ldots, k\}$ be a coloring of $\vec{C}_{2n+1}\langle n \rangle$. 
    If there is a color class $C_j^\alpha$ of $\alpha$ which is not a forbidden triangle,
    then $\alpha$ can be transformed into a coloring $\beta\colon \mathbb{Z}_{2n+1}\to \{1,2,\ldots, k\}$ such that its $j$th color class, $C_j^\beta$, is $C_j^\beta=\{a,a+1,a+2,\ldots,a+n-1\}$ for some $a\in \mathbb{Z}_{2n+1}$. 
    Moreover, in $\mathcal{D}_{k}(\vec{C}_{2n+1}\langle n \rangle)$, the colorings $\alpha$ and $\beta$ are at distance  at most $n-|C_j^\alpha|$ whenever $|C_j^\alpha|\in \{0,1\}$, and at most $n+2-|C_j^\alpha|$ otherwise.
    \label{lemma: extend a color class}
\end{lemma}
\begin{proof}
    Let $T=\vec{C}_{2n+1}\langle n \rangle$.
    Since $T$ is 3-dichromatic, $\mathcal{D}_{k}(T)$ is defined for every $k\geq 3$. 
    Let $C_1^\alpha$, $C_2^\alpha$, \ldots, $C_k^\alpha$ be the color classes of $\alpha$. 
    And, suppose w.l.o.g. that $|C_1^\alpha|$ is not a forbidden triangle.

\begin{case}
    $|C_1^\alpha|\geq 3$.
\end{case}
    Then the subdigraph of $T$ induced by $C_1^\alpha$, say $H$, is a transitive tournament, and thus it has a source. Let $a$ be the source of $H$, then $V(H)\subseteq \{a\}\cup N^+(a)=\{a,a+1,a+2,\ldots,a+n+1\}\setminus \{a+n\}$. Notice that $C_1^\alpha$ cannot contain a forbidden triangle. Hence, $\{a+n+1, a+2n+1, a+2n+2\}=\{a+n+1, a, a+1\}\not \subseteq V(H)$ and thus $|\{a+n+1, a+1\}\cap V(H)|\leq 1$. 
 
    If $a+n+1\notin V(H)$, then recolor one by one all vertices in $\{a,a+1,a+2,\ldots,a+n-1\}\setminus V(H)$ with color $1$. The current coloring $\beta$ satisfies $d(\alpha, \beta)\leq n-|C_1^\alpha|$.
    
    So, we assume that $a+n+1\in V(H)$.
    Recolor one by one all vertices in $\{a+2,a+3,\ldots,a+n-1\}\setminus V(H)$ with color $1$, this process takes at most $n-2-(|C_1^\alpha|-2)=n-|C_1^\alpha|$ vertex recolorings. 

    If $k\geq 4$, then let $j$ be a color in $\{1,2,\ldots, k\}\setminus\{1,\alpha(a+1),\alpha(a+n)\}$, then $C_j^\alpha\subseteq \{a+n+2,a+n+3,\ldots,a+2n\}$. By Lemma \ref{lemma: n volteada aciclico de n}, the subdigraph of $T$ induced by $\{a+n+1,a+n+2, a+n+3, \ldots, a+2n\}$ is a transitive tournament. Hence, we recolor vertex $a+n+1$ with color $j$. Next, recolor vertex $a+1$ with color $1$. The current coloring $\beta$ satisfies $d(\alpha,\beta)\leq n-|C_1^\alpha|+2=n+2-|C_1^\alpha|$ from $\alpha$.
    
    If $k=3$.
    
    Whenever $\alpha(a+1)=\alpha(a+n)=l$, let $h$ be the color in $\{1,2,3\}\setminus \{1,l\}$. Then $C_h^\alpha\subseteq \{a+n+2, a+n+3, \ldots, a+2n\}$. So, we recolor vertex $a+n+1$ with color $h$ and then recolor vertex $a+1$ with color $1$. The current coloring $\beta$ satisfies $d(\alpha,\beta)\leq n-|C_1^\alpha|+2=n+2-|C_1^\alpha|$ (see Figure \ref{figure: lemma 5 case 1.1}).

    \begin{figure}
\begin{center}
\begin{tikzpicture}[x=1cm,y=8mm, scale=0.9]
\draw[color=black] (-8.5,4) node {$\alpha\colon$};
\draw[color=black] (-8.5,1) node {$\beta\colon$};
\begin{scriptsize}
\draw [fill=ForestGreen] (-8,4) circle (2.5pt);
\draw[color=ForestGreen] (-8,4.43) node {$a$};
\draw [fill=red] (-7,4) circle (2.5pt);
\draw[color=red] (-7,4.43) node {$a+1$};
\draw [fill=black] (-6,4) circle (2.5pt);
\draw[color=black] (-6,4.43) node {$a+2$};
\draw[color=black] (-5,4) node {$\ldots$};
\draw [fill=ForestGreen] (-4,4) circle (2.5pt);
\draw[color=ForestGreen] (-4,4.43) node {$a+i$};
\draw[color=black] (-3,4) node {$\ldots$};
\draw [fill=black] (-2,4) circle (2.5pt);
\draw[color=black] (-2.1,4.43) node {$a+n-1$};
\draw [fill=red] (-1,4) circle (2.5pt);
\draw[color=red] (-1,4.43) node {$a+n$};
\draw [fill=ForestGreen] (0,4) circle (2.5pt);
\draw[color=ForestGreen] (0.1,4.43) node {$a+n+1$};
\draw[color=black] (1,4) node {$\ldots$};
\draw [fill=blue] (2,4) circle (2.5pt);
\draw[color=blue] (2,4.43) node {$a+n+i'$};
\draw [fill=black] (3,4) node  {$\ldots$};
\draw [fill=black] (4,4) circle (2.5pt);
\draw[color=black] (4,4.43) node {$a+2n$};
\draw [fill=ForestGreen] (-8,3) circle (2.5pt);
\draw[color=ForestGreen] (-8,3.43) node {$a$};
\draw [fill=red] (-7,3) circle (2.5pt);
\draw[color=red] (-7,3.43) node {$a+1$};
\draw [fill=ForestGreen] (-6,3) circle (2.5pt);
\draw[color=ForestGreen] (-6,3.43) node {$a+2$};
\draw[color=black] (-5,3) node {$\ldots$};
\draw [fill=ForestGreen] (-4,3) circle (2.5pt);
\draw[color=ForestGreen] (-4,3.43) node {$a+i$};
\draw[color=black] (-3,3) node {$\ldots$};
\draw [fill=ForestGreen] (-2,3) circle (2.5pt);
\draw[color=ForestGreen] (-2.1,3.43) node {$a+n-1$};
\draw [fill=red] (-1,3) circle (2.5pt);
\draw[color=red] (-1,3.43) node {$a+n$};
\draw [fill=ForestGreen] (0,3) circle (2.5pt);
\draw[color=ForestGreen] (0.1,3.43) node {$a+n+1$};
\draw[color=black] (1,3) node {$\ldots$};
\draw [fill=blue] (2,3) circle (2.5pt);
\draw[color=blue] (2,3.43) node {$a+n+i'$};
\draw [fill=black] (3,3) node {$\ldots$};
\draw [fill=black] (4,3) circle (2.5pt);
\draw[color=black] (4,3.43) node {$a+2n$};
\draw [fill=ForestGreen] (-8,2) circle (2.5pt);
\draw[color=ForestGreen] (-8,2.43) node {$a$};
\draw [fill=red] (-7,2) circle (2.5pt);
\draw[color=red] (-7,2.43) node {$a+1$};
\draw [fill=ForestGreen] (-6,2) circle (2.5pt);
\draw[color=ForestGreen] (-6,2.43) node {$a+2$};
\draw[color=black] (-5,2) node {$\ldots$};
\draw [fill=ForestGreen] (-4,2) circle (2.5pt);
\draw[color=ForestGreen] (-4,2.43) node {$a+i$};
\draw[color=black] (-3,2) node {$\ldots$};
\draw [fill=ForestGreen] (-2,2) circle (2.5pt);
\draw[color=ForestGreen] (-2.1,2.43) node {$a+n-1$};
\draw [fill=red] (-1,2) circle (2.5pt);
\draw[color=red] (-1,2.43) node {$a+n$};
\draw [fill=blue] (0,2) circle (2.5pt);
\draw[color=blue] (0.1,2.43) node {$a+n+1$};
\draw[color=black] (1,2) node {$\ldots$};
\draw [fill=blue] (2,2) circle (2.5pt);
\draw[color=blue] (2,2.43) node {$a+n+i'$};
\draw [fill=black] (3,2) node {$\ldots$};
\draw [fill=black] (4,2) circle (2.5pt);
\draw[color=black] (4,2.43) node {$a+2n$};
\draw [fill=ForestGreen] (-8,1) circle (2.5pt);
\draw[color=ForestGreen] (-8,1.43) node {$a$};
\draw [fill=ForestGreen] (-7,1) circle (2.5pt);
\draw[color=ForestGreen] (-7,1.43) node {$a+1$};
\draw [fill=ForestGreen] (-6,1) circle (2.5pt);
\draw[color=ForestGreen] (-6,1.43) node {$a+2$};
\draw[color=black] (-5,1) node {$\ldots$};
\draw [fill=ForestGreen] (-4,1) circle (2.5pt);
\draw[color=ForestGreen] (-4,1.43) node {$a+i$};
\draw[color=black] (-3,1) node {$\ldots$};
\draw [fill=ForestGreen] (-2,1) circle (2.5pt);
\draw[color=ForestGreen] (-2.1,1.43) node {$a+n-1$};
\draw [fill=red] (-1,1) circle (2.5pt);
\draw[color=red] (-1,1.43) node {$a+n$};
\draw [fill=blue] (0,1) circle (2.5pt);
\draw[color=blue] (0.1,1.43) node {$a+n+1$};
\draw[color=black] (1,1) node {$\ldots$};
\draw [fill=blue] (2,1) circle (2.5pt);
\draw[color=blue] (2,1.43) node {$a+n+i'$};
\draw [fill=black] (3,1) node {$\ldots$};
\draw [fill=black] (4,1) circle (2.5pt);
\draw[color=black] (4,1.43) node {$a+2n$};

\end{scriptsize}
\end{tikzpicture}
\caption{Case 1. $|C_1^{\alpha}| \geq 3$ whenever $k=3$ and $\alpha(a+1)=\alpha(a+n)$. Color 1 is green, color $l$ is red and color $h$ is blue. The vertices for which we do not know the color assigned to them are represented in black.}
\label{figure: lemma 5 case 1.1}
\end{center}
\end{figure}
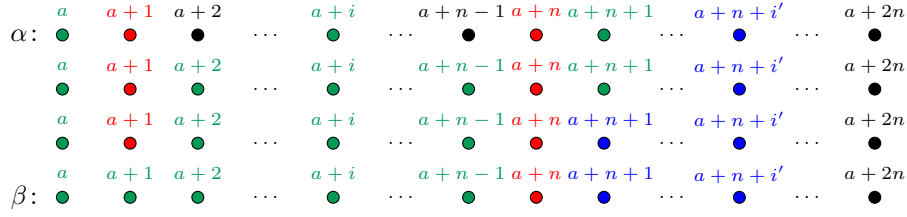

    Whenever $l=\alpha(a+1)\neq\alpha(a+n)=h$, then $\alpha(a+2n)=l$.  
    Otherwise, whenever $\alpha(a+2n)=h$ and if, for some $i\in \{a+n+2,a+n+3,\ldots,a+2n-1\}$, $\alpha(i)=h$, then $(a+n,i,a+2n,a+n)$ is a monochromatic cycle and if, for each $i\in \{a+n+2,a+n+3,\ldots,a+2n-1\}$, $\alpha(i)=l$, then $(a+n+2,a+2n-1,a+1,a+n+2)$ is a monochromatic cycle. In any case, we have a contradiction. Hence, $\alpha(a+2n)=l$. 
    The above implies that $C_h^\alpha\subseteq \{a+n,a+n+1,\ldots, a+2n-1\}$. By Lemma \ref{lemma: n volteada aciclico de n}, the subdigraph of $T$ induced by $\{a+n,a+n+1, a+n+2, \ldots, a+2n-1\}$ is a transitive tournament. So, we recolor vertex $a+n+1$ with color $h$. Next, recolor vertex $a+1$ with color $1$. The current coloring $\beta$ satisfies $d(\alpha,\beta)\leq n-|C_1^\alpha|+2=n+2-|C_1^\alpha|$ (see Figure \ref{figure: lemma 5 case 1.2}).

\begin{figure}
\begin{center}
\begin{tikzpicture}[x=1cm,y=8mm, scale=.9]
\draw[color=black] (-8.5,5) node {$\alpha\colon$};
\draw[color=black] (-8.5,2) node {$\beta\colon$};
\begin{scriptsize}
\draw [fill=ForestGreen] (-8,5) circle (2.5pt);
\draw[color=ForestGreen] (-8,5.43) node {$a$};
\draw [fill=red] (-7,5) circle (2.5pt);
\draw[color=red] (-7,5.43) node {$a+1$};
\draw [fill=black] (-6,5) circle (2.5pt);
\draw[color=black] (-6,5.43) node {$a+2$};
\draw[color=black] (-5,5) node {$\ldots$};
\draw [fill=ForestGreen] (-4,5) circle (2.5pt);
\draw[color=ForestGreen] (-4,5.43) node {$a+i$};
\draw[color=black] (-3,5) node {$\ldots$};
\draw [fill=black] (-2,5) circle (2.5pt);
\draw[color=black] (-2.1,5.43) node {$a+n-1$};
\draw [fill=blue] (-1,5) circle (2.5pt);
\draw[color=blue] (-1,5.43) node {$a+n$};
\draw [fill=ForestGreen] (0,5) circle (2.5pt);
\draw[color=ForestGreen] (0.1,5.43) node {$a+n+1$};
\draw[color=black] (1,5) node {$\ldots$};
\draw [fill=black] (2,5) circle (2.5pt);
\draw[color=black] (2,5.43) node {$a+n+i'$};
\draw [fill=black] (3,5) node  {$\ldots$};
\draw [fill=red] (4,5) circle (2.5pt);
\draw[color=red] (4,5.43) node {$a+2n$};
\draw [fill=ForestGreen] (-8,4) circle (2.5pt);
\draw[color=ForestGreen] (-8,4.43) node {$a$};
\draw [fill=red] (-7,4) circle (2.5pt);
\draw[color=red] (-7,4.43) node {$a+1$};
\draw [fill=ForestGreen] (-6,4) circle (2.5pt);
\draw[color=ForestGreen] (-6,4.43) node {$a+2$};
\draw[color=black] (-5,4) node {$\ldots$};
\draw [fill=ForestGreen] (-4,4) circle (2.5pt);
\draw[color=ForestGreen] (-4,4.43) node {$a+i$};
\draw[color=black] (-3,4) node {$\ldots$};
\draw [fill=ForestGreen] (-2,4) circle (2.5pt);
\draw[color=ForestGreen] (-2.1,4.43) node {$a+n-1$};
\draw [fill=blue] (-1,4) circle (2.5pt);
\draw[color=blue] (-1,4.43) node {$a+n$};
\draw [fill=ForestGreen] (0,4) circle (2.5pt);
\draw[color=ForestGreen] (0.1,4.43) node {$a+n+1$};
\draw[color=black] (1,4) node {$\ldots$};
\draw [fill=black] (2,4) circle (2.5pt);
\draw[color=black] (2,4.43) node {$a+n+i'$};
\draw [fill=black] (3,4) node  {$\ldots$};
\draw [fill=red] (4,4) circle (2.5pt);
\draw[color=red] (4,4.43) node {$a+2n$};
\draw [fill=ForestGreen] (-8,3) circle (2.5pt);
\draw[color=ForestGreen] (-8,3.43) node {$a$};
\draw [fill=red] (-7,3) circle (2.5pt);
\draw[color=red] (-7,3.43) node {$a+1$};
\draw [fill=ForestGreen] (-6,3) circle (2.5pt);
\draw[color=ForestGreen] (-6,3.43) node {$a+2$};
\draw[color=black] (-5,3) node {$\ldots$};
\draw [fill=ForestGreen] (-4,3) circle (2.5pt);
\draw[color=ForestGreen] (-4,3.43) node {$a+i$};
\draw[color=black] (-3,3) node {$\ldots$};
\draw [fill=ForestGreen] (-2,3) circle (2.5pt);
\draw[color=ForestGreen] (-2.1,3.43) node {$a+n-1$};
\draw [fill=blue] (-1,3) circle (2.5pt);
\draw[color=blue] (-1,3.43) node {$a+n$};
\draw [fill=blue] (0,3) circle (2.5pt);
\draw[color=blue] (0.1,3.43) node {$a+n+1$};
\draw[color=black] (1,3) node {$\ldots$};
\draw [fill=black] (2,3) circle (2.5pt);
\draw[color=black] (2,3.43) node {$a+n+i'$};
\draw [fill=black] (3,3) node {$\ldots$};
\draw [fill=red] (4,3) circle (2.5pt);
\draw[color=red] (4,3.43) node {$a+2n$};
\draw [fill=ForestGreen] (-8,2) circle (2.5pt);
\draw[color=ForestGreen] (-8,2.43) node {$a$};
\draw [fill=ForestGreen] (-7,2) circle (2.5pt);
\draw[color=ForestGreen] (-7,2.43) node {$a+1$};
\draw [fill=ForestGreen] (-6,2) circle (2.5pt);
\draw[color=ForestGreen] (-6,2.43) node {$a+2$};
\draw[color=black] (-5,2) node {$\ldots$};
\draw [fill=ForestGreen] (-4,2) circle (2.5pt);
\draw[color=ForestGreen] (-4,2.43) node {$a+i$};
\draw[color=black] (-3,2) node {$\ldots$};
\draw [fill=ForestGreen] (-2,2) circle (2.5pt);
\draw[color=ForestGreen] (-2.1,2.43) node {$a+n-1$};
\draw [fill=blue] (-1,2) circle (2.5pt);
\draw[color=blue] (-1,2.43) node {$a+n$};
\draw [fill=blue] (0,2) circle (2.5pt);
\draw[color=blue] (0.1,2.43) node {$a+n+1$};
\draw[color=black] (1,2) node {$\ldots$};
\draw [fill=black] (2,2) circle (2.5pt);
\draw[color=black] (2,2.43) node {$a+n+i'$};
\draw [fill=black] (3,2) node {$\ldots$};
\draw [fill=red] (4,2) circle (2.5pt);
\draw[color=red] (4,2.43) node {$a+2n$};
\end{scriptsize}
\end{tikzpicture}
\caption{Case 1. $|C_1^{\alpha}| \geq 3$ whenever $k=3$ and $\alpha(a+1)\neq \alpha(a+n)$. Color 1 is green, color $l$ is red and color $h$ is blue. The vertices for which we do not know the color assigned to them are represented in black.}
\label{figure: lemma 5 case 1.2}
\end{center}
\end{figure}
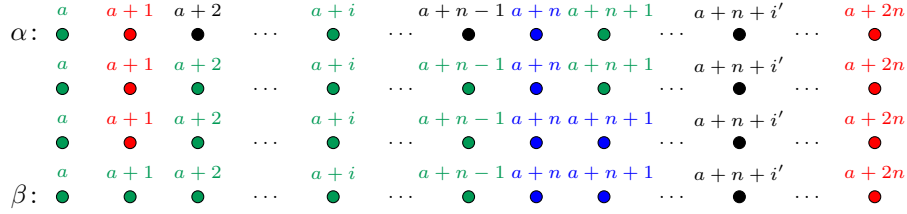

    Any coloring $\beta$ of Case 1 is at distance at most $n+2-|C_1^\alpha|$.

\begin{case}
    $|C_1^\alpha|\in\{0,1\}$.

\end{case}

    If $|C_1^\alpha|=1$, let $a$ be the only vertex in $C_1^\alpha$. Otherwise, let $a$ be any vertex in $T$. 
    Recolor one by one all vertices in $\{a,a+1,a+2, \ldots, a+n-1\}\setminus C_1^\alpha$ with color $1$. Thus, we obtain a coloring $\beta$ with the color class$C_1^\beta=\{a,a+1,a+2,\ldots,a+n-1\}$ in $n$ steps, and $d(\alpha,\beta) \leq n-|C_1^\alpha|$. 
       
\begin{case}
    $|C_1^\alpha|=2$. 
\end{case}

    Whenever $C_1^\alpha\subseteq \{a,a+1,a+2, \ldots, a+n-1\}$ for some $a\in \mathbb{Z}_{2n+1}$. 
    Then we recolor the vertices in $\{a,a+1,a+2, \ldots, a+n-1\}\setminus C_1^\alpha$ with color $1$. In this way, we obtain a coloring $\beta$ with the color class $C_1^\beta=\{a,a+1,a+2,\ldots,a+n-1\}$ in $n-2=n-|C_1^\alpha|$ steps, and  $d(\alpha,\beta)\leq n-|C_1^\alpha|$.
    Whenever $C_1^\alpha= \{a,a+n\}$ for some $a\in \mathbb{Z}_{2n+1}$, we have that $(a+n,a)\in A(T)$.
%
    Since $n\geq 4$, we may recolor vertex $a+n+2$  with color 1. The new coloring $\alpha'$ is acyclic as the subdigraph of $T$ induced by $\{a,a+n, a+n+2, a+2n-1\}$ is a transitive tournament. Moreover, $d(\alpha',\alpha)=1$, and its color class $C_1^{\alpha'}$ satisfies $|C_1^{\alpha'}|=3$, so we can proceed as in Case 1 to obtain a coloring $\beta$ such that $d(\alpha',\beta) \leq n+2-|C_1^{\alpha'}|=n+2-3=n+1-2=n+1-|C_1^\alpha|$ and, for some $a\in \mathbb{Z}_{2n+1}$,  $C^\beta_1=\{a,a+1,\ldots, a+n-1\}$. So, $d(\alpha, \beta)\leq n+1-|C_1^\alpha|+1=n+2-|C_1^\alpha|$. 
\end{proof}

Let $n$ be an integer, $n\geq 3$, and let $\mathscr{C}$ be the set of colorings  $\alpha\colon \mathbb{Z}_{2n+1}\to \{1,2,3\}$ of $\vec{C}_{2n+1}\langle n \rangle$ such that the color classes of $\alpha$ are $\{a\}$, $\{a+1, a+2, \ldots, a+n\}$ and $\{a+n+1,a+n+2,\ldots,a+2n\}$ for some $a\in \mathbb{Z}_{2n+1}$.

\begin{lemma}
    Let $n$ be an integer with $n\geq 3$. If $\alpha$ and $\beta$ are two 3-colorings in $\mathscr{C}$, then $\alpha$ and $\beta$ are at distance at most $3n-1$ in $\mathcal{D}_{3}(\vec{C}_{2n+1}\langle n \rangle)$. Moreover, if the singular color classes of $\alpha$ and $\beta$ have different colors, then $\alpha$ and $\beta$ are at distance at most $2n+1$.
    \label{lemma: grafica de recoloracion conexa para k=3 y n volteada 1}
   
\end{lemma}
\begin{proof}
    Let $\alpha$ and $\beta$ be two colorings in $\mathscr{C}$. Suppose w.l.o.g. that the color classes of $\alpha$ are $C_1^\alpha=\{0\}$, $C_2^\alpha=\{1,2,\ldots, n\}$ and $C_3^\alpha=\{n+1,n+2,\ldots, 2n\}$ and the color classes of $\beta$ are $C_h^\beta=\{a\}$, $C_i^\beta=\{a+1,a+2,\ldots, a+n\}$ and $C_j^\beta=\{a+n+1,a+n+2,\ldots, a+2n\}$, where $\{h,i,j\}=\{1,2,3\}$. 
    
\begin{case}
    $|C_1^\beta|=1$ and $\beta(0)=1$.
\end{case}   
    The above implies that $C_1^\beta=\{0\}$.
    If $i=2$, then $j=3$, and thus $\alpha$ and $\beta$ are the same coloring. So, we assume that $i=3$. In this way, the color classes of $\beta$ are $C_1^\beta=\{0\}$, $C_3^\beta=\{1,2,\ldots, n\}$ and $C_2^\beta=\{n+1,n+2,\ldots, 2n\}$.
    Start with coloring $\alpha$. Recolor, one by one, the vertices in $\{1,2,\ldots, n-1\}$ with color 1. The current coloring is acyclic by Lemma \ref{lemma: n volteada aciclico de n}. Then  recolor, one by one, the  vertices in $\{n+1,n+2,\ldots, 2n-1\}$ with color 2. After, recolor vertex $n$ with color 3. Next, recolor vertex $2n$ with color 2. Finally, recolor, one by one, all vertices in $\{1,2,\ldots, n-1\}$ with color 3. 
    The current coloring  $\beta$ satisfies $d(\alpha,\beta)\leq n-1+n-1+1+1+n-1 = 3n-1$.

\begin{case}
$|C_1^\beta|=1$ and $\beta(0)\neq 1$.
    
\end{case}

    The above implies that $C_1^\beta=\{x\}$ for some $x\in\{1,2,\ldots,2n\}$. By symmetry, we may assume that $x\in\{1,2,\ldots,n\}$. 

\begin{subcase}
    $x=n$.
\end{subcase}

    Start with coloring $\alpha$. If $\beta(1) = 2 = \alpha(n)$, then recolor vertex $n$ with color 1 and, after, recolor vertex $0$ with color 2. The current coloring $\beta$ satisfies $d(\alpha,\beta) = 2$.
   If $\beta(1) = 3\neq 2 = \alpha(n)$, then recolor all vertices in $\{n,n+1,\ldots, 2n-1\}\setminus\{n+1\}$ with color 1, this takes $n-1$ vertex recolorings. After, recolor vertex $n-1$ with color 3 and, then, recolor vertex $2n$ with color 2. Next, recolor all vertices in $\{0,1,\ldots,n-2\}\setminus \{1\}$ with color 3, this takes $n-2$ steps. Now, recolor vertex $n+1$ with color 2 and, then, recolor vertex 1 with color 3. Finally, recolor all vertices in $\{n+2,n+3, \ldots, 2n-1\}$ with color 2, this takes $n-2$ vertex recolorings. 
    The current coloring $\beta$ satisfies $d(\alpha,\beta)\leq n-1 + 2 + n-2 + 2 + n-2 = 3n-1$.
    
\begin{subcase}
    $1\leq x \leq n-1$.
\end{subcase}     

    Start with coloring $\alpha$. If $\beta(1) = 3\neq 2 = \alpha(x)$, then recolor all vertices in $\{1,2,\ldots, x\}$ with color 1, this takes $x$ vertex recolorings. After, recolor all vertices in $\{x+1,x+2, \ldots, x+n\}\setminus C_2^\alpha$ with color 2, this takes $n-(n-x)=x$ steps. Finally, recolor all vertices in $\{x+n+1,x+n+2, \ldots, x+2n\}\setminus C_3^\alpha$ with color 3, this also takes $x$ vertex recolorings. 
    The current coloring $\beta$ satisfies $d(\alpha,\beta)\leq 3x\leq 3(n-1)=3n-3$.
    If $\beta(1) = 2 = \alpha(x)$, then recolor all vertices in $\{x+n+1,x+n+2,\ldots, 2n\}$ with color 1, this takes $n-x$ vertex recolorings. After, recolor all vertices in $\{x+1, \ldots, n\}$ with color 3, this takes $n-x$ steps. Then, recolor all vertices in $\{x+n+2,x+n+3, \ldots, 2n, 0\}$ with color 2, this also takes $n-x$ vertex recolorings. Then recolor vertex $x$ with color 1 and, finally, recolor vertex $x+n+1$ with color 2. 
    The current coloring $\beta$ satisfies $d(\alpha,\beta)\leq 3(n-x)+2\leq 3(n-1)+2=3n-1$.

 \begin{case}
     $|C_1^\beta|=n$ and $\beta(0)=1$.
 \end{case}
    Let $x=\min\{b\in \{1,2,\ldots,n\}\colon \beta(b)\neq 1\}$. Then $C_2^\beta \cup C_3^\beta =\{x, x+1, \ldots, x+n\}$ and $C_1^\beta=\{x+n+1,x+n+2,\ldots, x+2n\}$.
    Start with coloring $\alpha$. Let $\alpha'$ be the coloring with color classes $C_1^{\alpha'}=\{x+n+1, x+n+2, \ldots, x+2n\}$, $C_2^{\alpha'}=\{x, \ldots, n\}$ and $C_3^{\alpha'}=\{n+1,\ldots, x+n\}$ obtained from $\alpha$ by recoloring all vertices in $\{x+n+1,x+n+2,\ldots, x+2n\}\setminus\{0\}$ with color 1, then $d(\alpha,\alpha')=n-1$.

\begin{subcase}
     $\beta(x)=2$.
\end{subcase}
    Here $\alpha'(x)=\alpha(x)=2=\beta(x)$. Now, we transform $\alpha'$ into $\beta$.
    If $C_2^\beta=\{x\}$, then recolor all vertices in $\{x+1, x+2,\ldots, x+2n\}\setminus C_3^\beta$ with color 3. 
    This process takes at most $n-1$ vertex recolorings, and the current coloring is $\beta$. Hence, $d(\alpha, \beta) \leq n-1+n-1=2n-2$.
    If $C_2^\beta=\{x,x+1,\ldots, x+n-1\}$, then recolor all vertices in $\{n+1, \ldots, x+n-1\}$ with color 2. This process takes at most $n-1$ vertex recolorings, and the current coloring is $\beta$. Hence, $d(\alpha, \beta) \leq n-1+n-1=2n-2$.

\begin{subcase}
    $\beta(x)=3$.
\end{subcase}
    Here $\alpha'(x)=\alpha(x)=2\neq 3=\beta(x)$. Now, we transform $\alpha'$ into $\beta$.
    If $C_3^\beta=\{x\}$, then recolor each vertex in $\{n+1, \ldots, x+n-1\}$ with color 2, this process takes at most $n-1$ steps. After, recolor vertex $x$ with color 3. And, finally, recolor vertex $x+n$ with color 2. 
    The current coloring is $\beta$ and satisfies $d(\alpha,\beta)\leq n-1+n-1+1+1=2n$.
    If $C_3^\beta=\{x,x+1,\ldots, x+n-1\}$, then recolor all vertices in $\{x+1, \ldots, x+n\}\setminus C_3^{\alpha'}$ with color 3, this process takes at most $n-1$ steps. After, recolor vertex $x+n$ with color 2. And, finally, recolor vertex $x$ with color 3. The current coloring is $\beta$ and satisfies $d(\alpha,\beta)\leq n-1+n-1+1+1 = 2n$.
    
\begin{case}
$|C_1^\beta|=n$ and $\beta(0)\neq 1$.    
\end{case} 
    The above implies that $C_1^\beta=\{x+1, x+2, \ldots, x+n\}$ for some $x\in\{0,1,\ldots, n\}$.

\begin{subcase}
    $C_2^\beta=\{x\}$ or $C_3^\beta=\{x+n+1\}$.
\end{subcase}
    Whenever $C_2^\beta=\{x\}$, we have that $C_3^\beta=\{x+n+1,x+n+2,\ldots, x+2n\}$.  
    Start with coloring $\alpha$. If $x=0$, then recolor vertex $n$ with color 1. After, recolor vertex 0 with color 2. Next, recolor all vertices in $\{1,2,\ldots,n-1\}$ with color 1. The current coloring is $\beta$ and satisfies $d(\alpha,\beta)\leq 1+1+n-1=n+1$.
    If $x\neq 0$, we recolor vertex $n+1$ with color 1. After, we recolor vertex 0 with color 3. Next, we recolor all vertices in $\{x+1, x+2, \ldots, x+n\}\setminus \{n+1\}$ with color 1. And, finally, we recolor all vertices in $\{1,\ldots, x-1\}$ with color 3. 
    The current coloring is $\beta$ and satisfies $d(\alpha,\beta)\leq 1+1+n-1+x-1\leq 2n$.
   
Analogously, $d(\alpha, \beta) \leq 2n$ whenever $C_3^\beta=\{x+n+1\}$.

\begin{subcase}
    $C_2^\beta=\{x+n+1\}$ or $C_3^\beta=\{x\}$.
\end{subcase}
    Whenever $C_3^\beta=\{x\}$, we have that $C_2^\beta=\{x+n+1,x+n+2,\ldots, x+2n\}$. 
    Start with coloring $\alpha$. If $x=0$, then recolor all vertices in $\{1,2,\ldots, n-1\}$ with color 1. Then, recolor all vertices in $\{n+1, n+2, \ldots, 2n-1\}$ with color 2. Next, recolor vertex $0$ with color 3. After, recolor vertex $n$ with color 1. Finally, recolor vertex $2n$ with color 2. 
    The current coloring is $\beta$ and satisfies $d(\alpha,\beta)\leq n-1+n-1+1+1+1=2n+1$.
    If $x\neq 0$, then recolor vertex $n$ with color 1 and, after, recolor vertex $0$ with color 2. Then, recolor all vertices in $\{x+1, x+2, \ldots, x+n\}\setminus\{n\}$ with color 1. Next, recolor all vertices in $\{x+n+2,x+n+3,\ldots,2n\}$ with color 2. After, recolor vertex $x$ with color 3. Finally, recolor vertex $x+n+1$ with color 2. 
    The current coloring is $\beta$ and satisfies $d(\alpha,\beta)\leq 1+1+n-1+n-x-1+1+1\leq 2n+1$.

    Analogously, $d(\alpha, \beta) \leq 2n+1$ whenever $C_2^\beta=\{x+n+1\}$.
\end{proof}

\begin{theorem}
    Let $n$ be an integer, $n\geq 5$. Then $\mathcal{D}_{3}(\vec{C}_{2n+1}\langle n \rangle)$ is connected and has diameter at most $3n+4+\lfloor \frac{2n+1}{3}\rfloor$.
    \label{teo: grafica de recoloracion conexa para k=3 y n volteada}
\end{theorem}
\begin{proof}     
   Let $T=\vec{C}_{2n+1}\langle n \rangle$.
   Let $\rho, \sigma \colon \mathbb{Z}_{2n+1}\to \{1,2,3\}$ be two 3-colorings of $T$.
Let $C_1^\rho$, $C_2^\rho$, $C_3^\rho$ and $C_1^\sigma$, $C_2^\sigma$, $C_3^\sigma$  be the color classes of $\rho$ and $\sigma$, respectively. Since $2n+1\geq 11$, it follows that $|C_i^\rho|\geq \lceil \frac{2n+1}{3}\rceil\geq 4$ and $|C_j^\sigma|\geq \lceil \frac{2n+1}{3}\rceil\geq 4$, for some $i\in\{1,2,3\}$ and $j\in\{1,2,3\}$.

By Lemma \ref{lemma: extend a color class}, there are two 3-colorings $\rho'$ and $\sigma'$ such that its $i$th and $j$th color classes are $C_i^{\rho'}=\{a+1,a+2,\ldots,a+n\}$ and $C_j^{\sigma'}=\{b+1,b+2,\ldots,b+n\}$, respectively, for some $a\in \mathbb{Z}_{2n+1}$ and $b\in \mathbb{Z}_{2n+1}$. Moreover, $d(\rho,\rho')\leq n+2-|C_i^\rho|\leq n+2-\lceil \frac{2n+1}{3}\rceil$ and $d(\sigma',\sigma)\leq n+2-|C_j^\sigma|\leq n+2-\lceil \frac{2n+1}{3}\rceil$. 

\begin{case} $i=j$.\end{case}

Let $g$ and $h$ be the two colors in $\{1,2,3\}\setminus \{i\}$.

\begin{subcase} $\{a+n+1, a+2n+1\}\subseteq C_g^{\rho'}$.\end{subcase}

Hence, $\{a+n+2,a+n+3,\ldots,a+2n\}=C_h^{\rho'}$ and $\{a+n+1, a+2n+1\}= C_g^{\rho'}$, otherwise $T$ contains a monochromatic directed triangle. Then, from $\rho'$, we obtain a coloring $\alpha\in \mathscr{C}$ by recoloring vertex $a+n+1$ with color $h$. Where $\alpha$ is such that $d(\rho,\alpha)\leq n+2-\lceil \frac{2n+1}{3}\rceil+1$ and whose singleton color class is colored with color $g$.

Now, let us consider coloring $\sigma'$, we have that the other color classes of $\sigma'$, $C_g^{\sigma'}$ and $C_h^{\sigma'}$, satisfy that $C_g^{\sigma'} \cup C_h^{\sigma'} =\{b+n+1, b+n+2, \ldots, b+2n+1\}$ and both are nonempty.

Whenever $\{b+n+1, b+2n+1\}\subseteq C_h^{\sigma'}$, it follows that $C_g^{\sigma'} =\{b+n+2, b+n+3, \ldots, b+2n\}$. So, we recolor vertex $b+2n+1$ with color $g$, and thus the current coloring $\beta$ is in $\mathscr{C}$ and $d(\beta,\sigma)\leq n+2-\lceil \frac{2n+1}{3}\rceil+1$. Moreover, by Lemma \ref{lemma: grafica de recoloracion conexa para k=3 y n volteada 1}, $d(\alpha,\beta)\leq 2n+1$, and thus $d(\rho, \sigma)\leq d(\rho,\alpha)+d(\alpha,\beta)+d(\beta,\sigma)\leq n+3-\lceil \frac{2n+1}{3}\rceil + 2n+1 + n+3-\lceil \frac{2n+1}{3}\rceil = 4n+7-2\lceil \frac{2n+1}{3}\rceil$.

Whenever $\{b+n+1, b+2n+1\}\subseteq C_g^{\sigma'}$, it follows that $C_h^{\sigma'} =\{b+n+2, b+n+3, \ldots, b+2n\}$. So, we recolor vertex $b+2n+1$ with color $h$, and thus the current coloring $\beta$ is in $\mathscr{C}$ and $d(\beta,\sigma)\leq n+2-\lceil \frac{2n+1}{3}\rceil+1$. Moreover, by Lemma \ref{lemma: grafica de recoloracion conexa para k=3 y n volteada 1}, $d(\alpha,\beta)\leq 3n-1$, and thus $d(\rho, \sigma)\leq d(\rho,\alpha)+d(\alpha,\beta)+d(\beta,\sigma)\leq n+3-\lceil \frac{2n+1}{3}\rceil + 3n-1 + n+3-\lceil \frac{2n+1}{3}\rceil = 5n+5-2\lceil \frac{2n+1}{3}\rceil$.

Whenever $\sigma'(b+n+1)\neq\sigma'(b+2n+1)$, we may suppose w.l.o.g. that $\sigma'(b+n+1)=g$ and $\sigma'(b+2n+1)=h$. So, we recolor all vertices in $\{b+n+2, b+n+3, \ldots, b+2n\}\setminus C_g^{\sigma'}$ with color $g$, this takes at most $n-1$ vertex recolorings and the current coloring $\beta$ is in $\mathscr{C}$ and $d(\beta,\sigma)\leq n+2-\lceil \frac{2n+1}{3}\rceil + n-1$. Moreover, by Lemma \ref{lemma: grafica de recoloracion conexa para k=3 y n volteada 1}, $d(\alpha,\beta)\leq 2n+1$, and thus $d(\rho, \sigma)\leq d(\rho,\alpha)+d(\alpha,\beta)+d(\beta,\sigma)\leq n+3-\lceil \frac{2n+1}{3}\rceil + 2n+1 + 2n+1-\lceil \frac{2n+1}{3}\rceil = 5n+5-2\lceil \frac{2n+1}{3}\rceil$.

In a similar way it can be proved that if $\{a+n+1, a+2n+1\}\subseteq C_h^{\rho'}$, $\{b+n+1, b+2n+1\}\subseteq C_g^{\sigma'}$, or $\{b+n+1, b+2n+1\}\subseteq C_h^{\sigma'}$, then $d(\rho,\sigma)\leq 5n+5-2\lceil\frac{2n+1}{3}\rceil$. 

\begin{subcase} $\rho'(a+n+1)\neq\rho'(a+2n+1)$ and $\sigma'(b+n+1)\neq\sigma'(b+2n+1)$.\end{subcase}

We assume w.l.o.g. that $\rho'(a+n+1)=\sigma'(b+n+1)=g$ and $\rho'(a+2n+1)=\sigma'(b+2n+1)=h$.

Since there are other $n-1$ vertices, namely $a+n+2,a+n+3,\ldots,a+2n$, colored with colors $g$ and $h$ in $\rho'$, one of the two color classes has at least $\lceil \frac{n-1}{2}\rceil+1$ vertices. The same can be said about the color classes of colors $g$ and $h$ in $\sigma'$.
Let $M=\max \{|C_{g}^{\rho'}|,|C_{h}^{\rho'}|,|C_{g}^{\sigma'}|,|C_{h}^{\sigma'}|\}$. 
Suppose w.l.o.g. that $M=|C_{g}^{\rho'}|$. Then $M=|C_{g}^{\rho'}| \geq \lceil \frac{n-1}{2}\rceil+1$, and $|C_{h}^{\rho'}|\leq n+1-M$. 
Start with $\rho'$ and recolor all vertices in $\{a+n+2,a+n+3,\ldots,a+2n\}\setminus C_g^{\rho'}$ with color $g$, this takes at most $n-M$ steps, and the current coloring $\alpha$ is in $\mathscr{C}$ and $d(\rho,\alpha)\leq n+2-\lceil\frac{2n+1}{3}\rceil + n-M$.

If $|C_{h}^{\sigma'}|\geq \lceil \frac{n-1}{2}\rceil+1$, then $M\geq |C_{h}^{\sigma'}|$ and $|C_{g}^{\sigma'}|\leq n+1-M$.
Start with coloring $\sigma'$ and recolor all vertices in $\{b+n+2,b+n+3,\ldots,b+2n\}\setminus C_h^{\sigma'}$ with color $h$, this takes at most $n-M$ steps, and the current coloring $\beta$ is in $\mathscr{C}$ and $d(\beta,\sigma)\leq n+2-\lceil\frac{2n+1}{3}\rceil + n-M$. 
Moreover, by Lemma \ref{lemma: grafica de recoloracion conexa para k=3 y n volteada 1}, $d(\alpha,\beta)\leq 2n+1$, and thus 
$d(\rho, \sigma)\leq d(\rho,\alpha)+d(\alpha,\beta)+d(\beta,\sigma)\leq 2n+2-\lceil\frac{2n+1}{3}\rceil- M + 2n+1 + 2n+2-\lceil\frac{2n+1}{3}\rceil - M = 6n+5-2\lceil\frac{2n+1}{3}\rceil-2M$. Since $M\geq \lceil \frac{n-1}{2}\rceil+1$, it follows that $n-M\leq n-\lceil \frac{n-1}{2}\rceil-1=\lfloor\frac{n-1}{2}\rfloor$. Hence, $d(\rho, \sigma)\leq 6n+5-2\lceil\frac{2n+1}{3}\rceil-2M \leq 4n+5-2\lceil\frac{2n+1}{3}\rceil+2\lfloor\frac{n-1}{2}\rfloor \leq 4n+5-2\lceil\frac{2n+1}{3}\rceil+ n-1=5n+4-2\lceil\frac{2n+1}{3}\rceil$.

If $|C_{g}^{\sigma'}|\geq \lceil \frac{n-1}{2}\rceil+1$, then $M\geq|C_{g}^{\sigma'}|$. Now, start with coloring $\sigma'$ and recolor all vertices in $\{b+n+2,b+n+3,\ldots,b+2n\}\setminus C_h^{\sigma'}$ with color $h$, the current coloring $\beta$ is in $\mathscr{C}$ and $d(\beta,\sigma)\leq n+2-\lceil\frac{2n+1}{3}\rceil + M-1$. Moreover, by Lemma \ref{lemma: grafica de recoloracion conexa para k=3 y n volteada 1}, $d(\alpha,\beta)\leq 2n+1$, and thus $d(\rho, \sigma)\leq d(\rho,\alpha)+d(\alpha,\beta)+d(\beta,\sigma)\leq 2n+2-\lceil\frac{2n+1}{3}\rceil - M 
+ 2n+1 + n+2-\lceil\frac{2n+1}{3}\rceil + M-1 = 5n + 4 - 2\lceil\frac{2n+1}{3}\rceil$. 

\begin{case}
    $i\neq j$.
\end{case} 

Let $h$ be the color in $\{1,2,3\}\setminus\{i,j\}$.

\begin{subcase} $\{a+n+1,a+2n+1\}\subseteq C_j^{\rho'}$.\end{subcase}

Hence, $\{a+n+2,a+n+3,\ldots,a+2n\}=C_h^{\rho'}$ and $\{a+n+1, a+2n+1\}= C_j^{\rho'}$. 
Then, from $\rho'$, we obtain a coloring $\alpha\in \mathscr{C}$ by recoloring vertex $a+n+1$ with color $h$. Where $\alpha$ is such that $d(\rho,\alpha)\leq n+2-\lceil \frac{2n+1}{3}\rceil+1$ and whose singleton color class is colored with color $j$.

Now, let us consider coloring $\sigma'$, we have that $C_i^{\sigma'}$ and $C_h^{\sigma'}$ satisfy that $C_i^{\sigma'} \cup C_h^{\sigma'} =\{b+n+1, b+n+2, \ldots, b+2n+1\}$ and both are nonempty.

Whenever $\{b+n+1, b+2n+1\}\subseteq C_h^{\sigma'}$, it follows that $C_i^{\sigma'} =\{b+n+2, b+n+3, \ldots, b+2n\}$. So, we recolor vertex $b+2n+1$ with color $i$, and thus the current coloring $\beta$ is in $\mathscr{C}$, its singleton color class is colored with color $h$ and $d(\beta,\sigma)\leq n+2-\lceil \frac{2n+1}{3}\rceil+1$. Moreover, by Lemma \ref{lemma: grafica de recoloracion conexa para k=3 y n volteada 1}, $d(\alpha,\beta)\leq 2n+1$, and thus $d(\rho, \sigma)\leq d(\rho,\alpha)+d(\alpha,\beta)+d(\beta,\sigma)\leq n+3-\lceil \frac{2n+1}{3}\rceil + 2n+1 + n+3-\lceil \frac{2n+1}{3}\rceil = 4n+7-2\lceil \frac{2n+1}{3}\rceil$.

Whenever $\{b+n+1, b+2n+1\}\subseteq C_i^{\sigma'}$, it follows that $C_h^{\sigma'} =\{b+n+2, b+n+3, \ldots, b+2n\}$. So, we recolor vertex $b+2n+1$ with color $h$, and thus the current coloring $\beta$ is in $\mathscr{C}$, its singleton color class is colored with color $i$ and $d(\beta,\sigma)\leq n+2-\lceil \frac{2n+1}{3}\rceil+1$. Moreover, by Lemma \ref{lemma: grafica de recoloracion conexa para k=3 y n volteada 1}, $d(\alpha,\beta)\leq 2n+1$, and thus $d(\rho, \sigma)\leq d(\rho,\alpha)+d(\alpha,\beta)+d(\beta,\sigma)\leq n+3-\lceil \frac{2n+1}{3}\rceil + 2n+1 + n+3-\lceil \frac{2n+1}{3}\rceil = 4n+7-2\lceil \frac{2n+1}{3}\rceil$.

Whenever $\sigma'(b+n+1)\neq\sigma'(b+2n+1)$, we may suppose w.l.o.g. that $\sigma'(b+n+1)=i$ and $\sigma'(b+2n+1)=h$. If $|C_i^{\sigma'}|\geq |C_h^{\sigma'}|$, then we recolor all vertices in $\{b+n+2, b+n+3, \ldots, b+2n\}\setminus C_i^{\sigma'}$ with color $i$, this takes at most $\lfloor \frac{n-1}{2} \rfloor$ steps and the current coloring $\beta$ is in $\mathscr{C}$ and its singleton color class is colored with color $h$. And, if $|C_h^{\sigma'}|\geq |C_i^{\sigma'}|$, then we recolor all vertices in $\{b+n+2, b+n+3, \ldots, b+2n\}\setminus C_h^{\sigma'}$ with color $h$, this takes at most $\lfloor \frac{n-1}{2} \rfloor$ vertex recolorings and the current coloring $\beta$ is in $\mathscr{C}$ and its singleton color class is colored with color $i$. Either way, 
$d(\beta,\sigma)\leq n+2-\lceil \frac{2n+1}{3}\rceil + \lfloor \frac{n-1}{2} \rfloor$. Moreover, by Lemma \ref{lemma: grafica de recoloracion conexa para k=3 y n volteada 1}, $d(\alpha,\beta)\leq 2n+1$, and thus $d(\rho, \sigma)\leq d(\rho,\alpha)+d(\alpha,\beta)+d(\beta,\sigma)\leq n+3-\lceil \frac{2n+1}{3}\rceil + 2n+1 + n+2-\lceil \frac{2n+1}{3}\rceil + \lfloor \frac{n-1}{2} \rfloor = 4n+6-2\lceil \frac{2n+1}{3}\rceil + \lfloor \frac{n-1}{2} \rfloor$.

In a similar way it can be proved that if  $\{b+n+1, b+2n+1\}\subseteq C_i^{\sigma'}$, then $d(\rho,\sigma)\leq 4n+6-2\lceil \frac{2n+1}{3}\rceil + \lfloor \frac{n-1}{2} \rfloor$.

\begin{subcase} $\{a+n+1,a+2n+1\}\subseteq C_h^{\rho'}$.\end{subcase}

Hence, $\{a+n+2,a+n+3,\ldots,a+2n\}=C_j^{\rho'}$ and $\{a+n+1, a+2n+1\}= C_h^{\rho'}$. 
Then, from $\rho'$, we obtain a coloring $\alpha\in \mathscr{C}$ by recoloring vertex $a+n+1$ with color $j$. Where $\alpha$ is such that $d(\rho,\alpha)\leq n+2-\lceil \frac{2n+1}{3}\rceil+1$ and whose singleton color class is colored with color $h$.

Now, let us consider coloring $\sigma'$, we have that $C_i^{\sigma'}$ and $C_h^{\sigma'}$ satisfy that $C_i^{\sigma'} \cup C_h^{\sigma'} =\{b+n+1, b+n+2, \ldots, b+2n+1\}$ and both are nonempty.

Whenever $\{b+n+1, b+2n+1\}\subseteq C_h^{\sigma'}$, it follows that $C_i^{\sigma'} =\{b+n+2, b+n+3, \ldots, b+2n\}$. So, we recolor vertex $b+2n+1$ with color $i$, and thus the current coloring $\beta$ is in $\mathscr{C}$, its singleton color class is colored with color $h$ and $d(\beta,\sigma)\leq n+2-\lceil \frac{2n+1}{3}\rceil+1$. Moreover, by Lemma \ref{lemma: grafica de recoloracion conexa para k=3 y n volteada 1}, $d(\alpha,\beta)\leq 3n-1$, and thus $d(\rho, \sigma)\leq d(\rho,\alpha)+d(\alpha,\beta)+d(\beta,\sigma)\leq n+3-\lceil \frac{2n+1}{3}\rceil + 3n-1 + n+3-\lceil \frac{2n+1}{3}\rceil = 5n+5-2\lceil \frac{2n+1}{3}\rceil$.

Whenever $\{b+n+1, b+2n+1\}\subseteq C_i^{\sigma'}$, it follows that $C_h^{\sigma'} =\{b+n+2, b+n+3, \ldots, b+2n\}$. So, we recolor vertex $b+2n+1$ with color $h$, and thus the current coloring $\beta$ is in $\mathscr{C}$, its singleton color class is colored with color $i$ and $d(\beta,\sigma)\leq n+2-\lceil \frac{2n+1}{3}\rceil+1$. Moreover, by Lemma \ref{lemma: grafica de recoloracion conexa para k=3 y n volteada 1}, $d(\alpha,\beta)\leq 2n+1$, and thus $d(\rho, \sigma)\leq d(\rho,\alpha)+d(\alpha,\beta)+d(\beta,\sigma)\leq n+3-\lceil \frac{2n+1}{3}\rceil + 2n+1 + n+3-\lceil \frac{2n+1}{3}\rceil = 4n+7-2\lceil \frac{2n+1}{3}\rceil$.

Whenever $\sigma'(b+n+1)\neq\sigma'(b+2n+1)$, we may suppose w.l.o.g. that $\sigma'(b+n+1)=i$ and $\sigma'(b+2n+1)=h$. Then we recolor all vertices in $\{b+n+2, b+n+3, \ldots, b+2n\}\setminus C_i^{\sigma'}$ with color $i$, this takes at most $n-1$ steps and the current coloring $\beta$ is in $\mathscr{C}$ and its singleton color class is colored with color $h$. Hence, 
$d(\beta,\sigma)\leq n+2-\lceil \frac{2n+1}{3}\rceil + n-1$. Moreover, by Lemma \ref{lemma: grafica de recoloracion conexa para k=3 y n volteada 1}, $d(\alpha,\beta)\leq 2n+1$, and thus $d(\rho, \sigma)\leq d(\rho,\alpha)+d(\alpha,\beta)+d(\beta,\sigma)\leq n+3-\lceil \frac{2n+1}{3}\rceil + 2n+1 + 2n+1-\lceil \frac{2n+1}{3}\rceil = 5n+5-2\lceil \frac{2n+1}{3}\rceil$.

In a similar way it can be proved that if  $\{b+n+1, b+2n+1\}\subseteq C_i^{\sigma'}$, then $d(\rho,\sigma)\leq 5n+5-2\lceil \frac{2n+1}{3}\rceil$. 

\begin{subcase}
    $\rho'(a+n+1)\neq \rho'(a+2n+1)$ and $\sigma'(b+n+1)\neq \sigma'(b+2n+1)$.
\end{subcase}
We may assume that $\rho'(a+n+1)=j$, $\rho'(a+2n+1)=h$, $\sigma'(b+n+1)=i$, and $\sigma'(b+2n+1)=h$.
Suppose w.l.o.g. that $|C_h^{\rho'}|\geq |C_h^{\sigma'}|$. Then, starting from $\rho'$, recolor all vertices in $\{a+n+2, a+n+3, \ldots, a+2n\}\setminus C_h^{\rho'}$ with color $h$, this takes at most $n-1-|C_h^{\rho'}|$ vertex recolorings. The current coloring $\alpha$ is in $\mathscr{C}$ and its singleton color class is colored with color $j$. Hence, $d(\rho,\alpha)\leq n+2-\lceil \frac{2n+1}{3}\rceil + n-1-|C_h^{\rho'}|$.

Now, starting from $\sigma'$, recolor all vertices in $\{b+n+2, b+n+3, \ldots, b+2n\}\setminus C_i^{\rho'}$ with color $i$, this takes at most $|C_h^{\sigma'}|-1\leq |C_h^{\rho'}|-1$ vertex recolorings. The current coloring $\beta$ is in $\mathscr{C}$ and its singleton color class is colored with color $h$. Thus, $d(\beta,\sigma)\leq n+2-\lceil \frac{2n+1}{3}\rceil + |C_h^{\rho'}|-1$. 
Moreover, by Lemma \ref{lemma: grafica de recoloracion conexa para k=3 y n volteada 1}, $d(\alpha,\beta)\leq 2n+1$, and thus $d(\rho, \sigma)\leq d(\rho,\alpha)+d(\alpha,\beta)+d(\beta,\sigma)\leq 2n+1-\lceil \frac{2n+1}{3}\rceil -|C_h^{\rho'}| + 2n+1 + n+1-\lceil \frac{2n+1}{3}\rceil + |C_h^{\rho'}|=5n+3-2\lceil \frac{2n+1}{3}\rceil$. \vspace{1.7mm}

From all cases, it follows that $d(\rho,\sigma)\leq 5n+5-2\lceil \frac{2n+1}{3}\rceil \leq 5n+5-2\left( \frac{2n+1}{3}\right)=\frac{11n+13}{3}=3n+4+\frac{2n+1}{3}$. Hence, $\mathcal{D}_{3}(T)$ is connected and $\diam(\mathcal{D}_{3}(T))\leq 3n+4+\lfloor\frac{2n+1}{3}\rfloor$.
\end{proof}

Using a Python program, we obtained information on $\mathcal{D}_{k}(\vec{C}_{7}\langle 3 \rangle)$ for each $k\in\{3,4,5,6\}$, which is included in Table \ref{table}. 
Provided with all the acyclic subsets of $\vec{C}_{7}\langle 3 \rangle$ in which $0$ is the source and an empty set of edges $E$, the program performs the following steps:
\begin{enumerate}
    \item Make a list $L$ with every acyclic subset of $\vec{C}_{7}\langle 3 \rangle$ by rotating each acyclic subset in which $0$ is the source.
    \item Make a list $P$ with all the ordered partitions of $\{1,2,3,4,5,6,7\}$ with $k$ parts such that each part is either an element of $L$ or the empty set.
    \item Decide if two $k$-partitions ($k$-colorings), $P[i]$ and $P[j]$ with $i<j$, are adjacent in $\mathcal{D}_{k}(\vec{C}_{7}\langle 3 \rangle)$, and if they are, then include the pair $(i,j)$ to a list of edges $E$. 
    \item Construct the graph $\mathcal{D}_{k}(\vec{C}_{7}\langle 3 \rangle)$ with vertex set the first $\len(P)$ nonnegative integers and $E$ as the set of edges, and evaluate basic parameters of $\mathcal{D}_{k}(\vec{C}_{7}\langle 3 \rangle)$.
\end{enumerate}

\begin{table}[h!]
\begin{center}
\begin{tabular}{ |c|c|c|c|c|c|c|c|c| }
 \hline
& \multicolumn{8}{|c|}{$\mathcal{D}_{k}(\vec{C}_{7}\langle 3 \rangle)$}\\
 \hline
 $k$ & order & size & connected & $\delta$ & $\Delta$ & diameter & radius & girth\\
 \hline
 3 & 504 & 1,512 & $\checkmark$ & 6 & 6 & 8 & 7 & 3\\

 4 & 7,560 & 54,684 & $\checkmark$ & 13 & 15 & 8 & 7 & 3\\
 
 5 & 47,880 & 536,760 & $\checkmark$ & 20 & 24 & 8 & 7 & 3\\
 6 & 199,080 & 2,997,540 & $\checkmark$ & 27 & 33 & 8 & 7 & 3\\
 \hline
\end{tabular}
\end{center}
\caption{$\mathcal{D}_{k}(\vec{C}_{7}\langle 3 \rangle)$ for $k\in\{3,4,5,6\}$}
\label{table}
\end{table}

\begin{theorem}
    Let $k$ be an integer, with $ k\geq 3$. Then $\mathcal{D}_{k}(\vec{C}_{7}\langle 3 \rangle)$ is connected and has diameter 8.  
    \label{theorem: ST7}
\end{theorem}
\begin{proof}
    Let $T=\vec{C}_{7}\langle 3 \rangle$.
    We have that $\mathcal{D}_{k}(T)$ is connected and has diameter 8, for each $k\in\{3,4,5,6\}$, by the information in Table \ref{table}.
    Moreover, $\mathcal{D}_{k}(T)$ is connected for every $k\geq 7$, by Corollary \ref{corollary: n-mixing}.

    Let $\alpha$ and $\beta$ be two $k$-colorings of $T$ for some $k\geq 7$. Let $A=\{i\in \{1,2,\ldots,k\}\colon C_i^\alpha\neq \emptyset\}$ and $B=\{i\in \{1,2,\ldots,k\}\colon C_i^\beta\neq \emptyset\}$. Since $3\leq |A| \leq 7$ and $3\leq |B|\leq 7$, it follows that $3\leq |A \cup B|\leq 14$.

    If $6\leq |A|\leq 7$ or $6\leq |B|\leq 7$, we have that one of the colorigs has seven nonempty color classes, all singletons, or it has six nonempty color classes, five singletons and one with two vertices. Thus, by Lemmas \ref{lemma: k colors k vertices} and \ref{lemma: k colors k+1 vertices}, $d(\alpha,\beta)\leq 7$.
    So, we assume $3\leq |A| \leq 5$ and $3\leq |B|\leq 5$.
    
    If $|A\cup B|\leq 6$, then the subgraph of $\mathcal{D}_{k}(T)$ induced by the colorings that have empty color classes for each color in  $\{1,2,\ldots, k\}\setminus (AUB)$ is isomorphic to $\mathcal{D}_{h}(T)$, where $h=|A\cup B|$. Hence, $d(\alpha,\beta)\leq 8$. So, we assume $|A\cup B|\geq 7$.

    Moreover, if $|A\cap B|\leq 1$, then $\alpha$ and $\beta$ have at most one nonempty color class of the same color, say color $j$. Then, starting at $\alpha$ we recolor each vertex $u\in \mathbb{Z}_7\setminus C_j^\beta$ with color $\beta(u)$, and then give color $j$ to vertices in $C_j^\beta$. In this way,  $d(\alpha,\beta)\leq 7$. So, we assume $|A\cap B|\geq 2$. Hence, $4\leq |A| \leq 5$ and $4\leq |B|\leq 5$.

    Whenever $|A\cap B|= 2$, it follows that $\alpha$ or $\beta$ has 5 nonempty color classes. Suppose w.l.o.g. that $|B| = 5$, then $|B\setminus A|=3$. Let $j$ and $j'$ be the two colors in $A\cap B$.
    Start at $\alpha$. Recolor each vertex $u\in \mathbb{Z}_7\setminus C_j^\beta\cup C_{j'}^\beta$ with color $\beta(u)$. We may assume there are four vertices in $C_j^\beta\cup C_{j'}^\beta$, and that one of those four vertices, say $v$, is colored with a color in $B\setminus A$, otherwise we recolor one vertex, say $v$, with a color in $B\setminus A$ such that it is assigned to exactly one vertex. Then we have three vertices that need to be colored with colors $j$ and $j'$, this can be achieved in at most 3 steps, by Lemma \ref{lemma: k colors k+1 vertices}. After, recolor $v$ with $\beta(v)$. In this way, $d(\alpha,\beta)\leq 8$.

    Whenever $|A\cap B|= 3$, it follows that $|A| = |B| = 5$, then $|B\setminus A|=2$. Let $i$ and $i'$ be the two colors in $B\setminus A$.
    Start at $\alpha$. Recolor each vertex $u$ in $C_{i}^\beta\cup C_{i'}^\beta$ with color $\beta(u)$, then there are at most five vertices that need to be recolored with the three colors in $A\cap B=\{j,j',j''\}$. Observe that, in the current coloring $\gamma$, one of the color classes $C_h^\gamma$ with $h\in\{j, j', j''\}$ is empty or a singleton, suppose w.l.o.g. that $h=j$. Let us assume $C_{j}^\gamma \neq \emptyset$. If $C_{j}^\gamma=C_{j}^\beta$, the four remaining vertices can be recolored as above. So we assume $C_{j}^\gamma\neq C_{j}^\beta$, so we recolor some vertex in $C_{j}^\beta\setminus C_{j}^\gamma$ with color $j$. Now, one of the color classes of $j'$ or $j''$ is empty or a singleton. We proceed as before, until we end up with $\alpha$. Since at most one vertex is recolored twice, $d(\alpha,\beta)\leq 8$.
    
    Therefore, $\diam(\mathcal{D}_{k}(T))\leq 8$. Moreover,  let $\alpha, \beta\colon \mathbb{Z}_{7}\to \{1,2,3\}$ be the colorings given by $\alpha(0)=1$, $\alpha(1)=1$, $\alpha(2)=1$ $\alpha(3)=2$, $\alpha(4)=2$, $\alpha(5)=2$ and $\alpha(6)=3$, and $\beta(0)=2$, $\beta(1)=2$, $\beta(2)=2$ $\beta(3)=3$, $\beta(4)=1$, $\beta(5)=1$ and $\beta(6)=1$. It is easy to check that $d(\alpha, \beta)=8$ for every $k\geq 3$. Hence, $\diam(\mathcal{D}_{k}(T)=8$.
\end{proof}

\begin{theorem}
    Let $n$ and $k$ be two integers, $n\geq 4$ and $ k\geq 4$. Then $\mathcal{D}_{k}(\vec{C}_{2n+1}\langle n \rangle)$ is connected whenever $k\neq \frac{2n+1}{3}$, and has diameter at most $4n+2+\lfloor \frac{n}{2} \rfloor$.
    \label{theorem: D_k(n) k not multiple of 3}
\end{theorem}
\begin{proof}
    Let $T=\vec{C}_{2n+1}\langle n \rangle$.
    Since $T$ is 3-dichromatic, $\mathcal{D}_{k}(T)$ is defined for every $k\geq 4$.  
   Let $\rho, \sigma \colon \mathbb{Z}_{2n+1}\to \{1,2,\ldots,k\}$ be two colorings of $T$. And let $C_1^\rho$, $C_2^\rho$, \ldots, $C_k^\rho$ and $C_1^\sigma$, $C_2^\sigma$, \ldots, $C_k^\sigma$ be the color classes of $\rho$ and $\sigma$, respectively.

\begin{case}   $C_i^\rho=\emptyset$ for some $i\in \{1,2,\ldots, k\}$.
\end{case}
Supposse w.l.o.g. that $C_k^\rho=\emptyset$. 
Since $C_k^{\sigma}$ induces an acyclic tournament, it has a source. Assume w.l.o.g. that 0 is the source of this tournament, then $C_k^{\sigma}\subseteq \{0,1,2,\ldots, n+1\}\setminus\{n\}$. Start at $\rho$ and recolor each vertex in $\{0,1,2,\ldots, n-1\}$ with color $k$, this takes $n$ steps. Call $\rho'$ to the current coloring.

Whenever $\sigma(n+1)\neq k$.
If $\rho'(n)=\rho'(2n)$, then  recolor, if needed, one of this two vertices with the color assigned to it by $\sigma$. After, recolor each vertex $u$ in $\{n+1,n+2,\ldots,2n-1\}$ (again, if needed) with color $\sigma(u)$. Then, recolor the last vertex in $\{n,2n\}$.
Finally, recolor each vertex $u$ in $\{0,1,2,\ldots, n-1\}\setminus C_k^\sigma$ with color $\sigma(u)$. The total process takes at most $n+n+1+n-|C_k^\sigma|\leq 3n+1$.
If $\rho'(n)\neq \rho'(2n)$, then  recolor each vertex $u$ in $\{n+1,n+2,\ldots,2n-1\}$, if needed, with color $\sigma(u)$. Then, recolor the vertices in $\{n,2n\}$ with the color that $\sigma$ gives them.
Finally, recolor each vertex $u$ in $\{0,1,2,\ldots, n-1\}\setminus C_k^\sigma$ with color $\sigma(u)$. The total process takes at most $n+n+1+n-|C_k^\sigma|\leq 3n+1$.

Whenever $\sigma(n+1)= k$.
If $\rho'(n)=\rho'(2n)$, then recolor, if needed, one of this two vertices with the color assigned to it by $\sigma$. After, recolor each vertex $u$ in $\{n+2,n+3,\ldots,2n-1\}$ with color $\sigma(u)$. Next, if needed, recolor the other vertex in $\{n,2n\}$. Then, recolor vertex 1 with color $\sigma(1)$. After, recolor vertex $n+1$ with color $k$. Finally, recolor each vertex $u$ in $\{0,1,2,\ldots, n-1\}\setminus C_k^\sigma$ with color $\sigma(u)$. The total process takes at most $n+n+2+n-(|C_k^\sigma|-1)\leq 3n+1$.
If $\rho'(n)\neq \rho'(2n)$, then  recolor each vertex $u$ in $\{n+2,n+3,\ldots,2n-1\}$, if needed, with color $\sigma(u)$. Then, recolor the vertices in $\{n,2n\}$ with the color that $\sigma$ gives them. After, recolor vertex 1 with color $\sigma(1)$ and, next, recolor vertex $n+1$ with color $k$.
Finally, recolor each vertex $u$ in $\{0,1,2,\ldots, n-1\}\setminus C_k^\sigma$ with color $\sigma(u)$. The total process takes at most $n+n+2+n-(|C_k^\sigma|-1)\leq 3n+1$. 

\begin{case} There are two color classes $C_i^\rho$ and $C_j^\rho$, with $\{i,j\}\subseteq\{1,2,\ldots, k\}$ and $i\neq j$, such that $C_i^\rho \cup C_j^\rho \subseteq \{a+1,a+2,\ldots,a+n\}$ for some $a\in \mathbb{Z}_{2n+1}$.
\end{case}

Suppose w.l.o.g. that $|C_i^\rho|\geq |C_j^\rho|$. Start at $\rho$, then recolor all vertices in $C_j^\rho$ with color $i$. This process takes at most $\lfloor \frac{n}{2} \rfloor$ steps, and the new coloring $\rho'$ satisfies that $C_j^{\rho'}=\emptyset$, and thus $d(\rho',\sigma)\leq 3n+2$, by Case 1. Hence, $d(\rho,\sigma)\leq d(\rho,\rho')+d(\rho',\sigma) \leq 3n+2+\lfloor \frac{n}{2} \rfloor$.

\begin{case}
    For each pair of color classes $C_i^\rho$ and $C_j^\rho$, with $\{i,j\}\subseteq\{1,2,\ldots, k\}$ and $i\neq j$, there is no $a\in \mathbb{Z}_{2n+1}$ such that $C_i^\rho \cup C_j^\rho \subseteq \{a+1,a+2,\ldots,a+n\}$.
\end{case}

\begin{subcase} There is a color class $C_i^\rho$, with $i\in\{1,2,\ldots, k\}$, such that $C_i^\rho \subseteq \{a,a+1,a+2,\ldots,a+n-1\}$ and $C_i^\rho\neq \emptyset$, for some $a\in \mathbb{Z}_{2n+1}$. 
\end{subcase}

Then $C_i^\rho$ is not a forbidden triangle. 
Assume w.l.o.g. that $a=0$. So, if we start at $\rho$ we can recolor every vertex in $\{0,1,2\ldots,n-1\}\setminus C_i^\rho$ with color $i$, this process takes at most $n-1$ steps. Call $\rho'$ to the current coloring. It follows that there must be two color classes of $\rho'$  contained in $\{n,n+1, \ldots, 2n-1\}$ or in $\{n+1,n+2, \ldots, 2n\}$, so we can proceed as in Case 2. An thus, $d(\rho,\sigma)\leq n-1+ 3n+2+\lfloor \frac{n}{2} \rfloor \leq 4n+1+\lfloor \frac{n}{2} \rfloor$. 

\begin{subcase}
    For each color class $C_i^\rho$, with $i\in\{1,2,\ldots, k\}$, there exists an $a\in \mathbb{Z}_{2n+1}$ such that $\{a,a+n+1\}\subseteq C_i^\rho$.
\end{subcase}
    Since $k\neq \frac{2n+1}{3}$, it follows that, for some $s\in\{1,2,\ldots, k\}$, $|C_s^\rho|\neq 3$, and thus $\rho$ has a color class that is not a forbidden triangle.
    By Lemma \ref{lemma: extend a color class}, there is a coloring $\rho'\colon \mathbb{Z}_{2n+1}\to \{1,2,\ldots, k\}$ such that $C_s^{\rho'}=\{a,a+1,\ldots, a+n-1\}$  for some $a\in \mathbb{Z}_{2n+1}$ and $d(\rho,\rho')\leq n$.
    As in Case 3.1, there must be two color classes of $\rho'$  contained in $\{a+n,a+n+1, \ldots, n+2n-1\}$ or in $\{a+n+1,a+n+2, \ldots, a+2n\}$, so we can proceed as in Case 2. An thus, $d(\rho,\sigma)\leq n+ 3n+2+\lfloor \frac{n}{2} \rfloor \leq 4n+2+\lfloor \frac{n}{2} \rfloor$. 
\end{proof}

\begin{theorem}
    Let $n\geq 4$ such that $2n+1\equiv 0 \pmod{3}$. Then $\mathcal{D}_{\frac{2n+1}{3}}(\vec{C}_{2n+1}\langle n \rangle)$ consists of $3(\frac{2n+1}{3})!$ isolated vertices and one further component that has diameter at most $2n+1+\lfloor \frac{2n+1}{4}\rfloor$.
    \label{theorem: D_k(n) k a multiple of 3}
\end{theorem}
\begin{proof}
    Let $n'\geq 1$ be such that $2n+1=6n'+3$, then $\frac{2n+1}{3}=2n'+1$ and $n=3n'+1$. Let 
    $T=\vec{C}_{6n'+3}\langle 3n'+1 \rangle=\vec{C}_{2n+1}\langle n \rangle$.
    Since $T$ is 3-dichromatic, $\mathcal{D}_{2n'+1}(T)$ is defined for each $n'\geq 1$.    
   Let $\alpha \colon \mathbb{Z}_{6n'+3}\to \{1,2,\ldots,2n'+1\}$ be the coloring of $T$ given by: $$\alpha(i)=\begin{cases}
        1 & \text{whenever }i\in\{0,3n'+1,3n'+2\};\\
        2 & \text{whenever }i\in\{1,2,3n'+3\};\\
        3 & \text{whenever }i\in\{3,3n'+4,3n'+5\};\\
        4 & \text{whenever }i\in\{4,5,3n'+6\};\\
        \vdots&\qquad\qquad\qquad\vdots\\
        2n' & \text{whenever }i\in\{3n'-2,3n'-1,6n'\};\\
        2n'+1 & \text{whenever }i\in\{3n',6n'+1,6n'+2\}.\\
    \end{cases}$$
   By Lemma \ref{lemma: acyclic triangle in circulant tournament with (n+1,0)}, each color class induces an acyclic subdigraph. 
   Moreover, since each color class is a maximum acyclic set, if we recolor any vertex, then the new color function is not acyclic, and thus $\alpha$ is an isolated vertex in $\mathcal{D}_{2n'+1}(T)$.
   
   We denote by $\alpha+\ell$ the coloring obtained from alpha by rotating its color classes by $\ell$ places, this is, $\alpha_\ell$ be the coloring defined as $\alpha_\ell(i+\ell)=\alpha(i)$ for each $i\in \mathbb{Z}_{6n'+3}$.
   
   The colorings $\alpha_1$ and $\alpha_2$ are also isolated vertices in $\mathcal{D}_{2n'+1}(T)$. Moreover, every permutation of the colors assigned to the color classes of $\alpha$, $\alpha_1$, or $\alpha_2$, results in a new coloring which is also an isolated vertex in $\mathcal{D}_{2n'+1}(T)$. Thus, $\mathcal{D}_{2n'+1}(T)$ contains at least $3(2n'+1)!$ isolated vertices.  

   Now, suppose that $\rho$ and $\sigma$ are two colorings such that each of them contains a color class that is not a forbidden triangle.
   Proceeding as in the proof of Theorem \ref{theorem: D_k(n) k not multiple of 3}, with Cases 1, 2 and 3 (excluding Case 3.2), we obtain that $d(\rho,\sigma)\leq 4(3n'+1)+2+\lfloor \frac{3n'+1}{2}\rfloor = 12n'+6+\lfloor\frac{3n'+1}{2}\rfloor\leq 2n+1+\lfloor\frac{2n+1}{4}\rfloor$. Moreover, as any two colorings, each containing a color class that is not a forbidden triangle, belong to the same connected component, it follows that $\mathcal{D}_{2n'+1}(T)$ has precisely $3(2n'+1)!$ isolated vertices.
\end{proof}

\section*{Acknowledgments}

This research was supported by SECIHTI Basic Science Project CBF2023-2024-552 \emph{Network Optimization: Moore Graphs and Cages} and 
the first author received support for a postdoctoral position at Department of Applied Mathematics and Systems at Metropolitan Autonomous University Cuajimalpa, under the grant \emph{Estancias Posdoctorales por M\'exico} by SECIHTI (CVU: 662742).
\bibliographystyle{plain}
\bibliography{RecoloringBib}

\end{document}